\documentclass[final,leqno,onefignum,onetabnum]{siamltex1213}
\usepackage{amsfonts}
\usepackage{amssymb,amsbsy}
\usepackage{amsmath}
\usepackage{amsfonts, amscd}
\usepackage{algorithm, algorithmic}
\usepackage{color}
\usepackage{url}
\usepackage{verbatim}
\usepackage{setspace}
\usepackage[abs]{overpic}
\usepackage{graphicx}
\usepackage{subfigure}
\usepackage{bbm, bm}
\usepackage{accents}
\usepackage{mleftright}
\usepackage{dsfont}

\newtheorem{example*}{Example}
\numberwithin{equation}{section}
\renewcommand{\theequation}{\thesection.\arabic{equation}}

\newtheorem{remark}{Remark}

\newcommand{\bb}{\mathbb}
\newcommand{\mb}{\boldsymbol}
\renewcommand{\bm}{\boldsymbol}
\newcommand{\mc}{\mathcal}

\newcommand{\mcb}[1]{\mb{\mc{#1}}}
\newcommand{\norm}[2]{\left\| #1 \right\|_{#2}}
\newcommand{\reals}{\mathbb{R}}
\newcommand{\<}{\langle}
\renewcommand{\>}{\rangle}
\newcommand{\innerprod}[2]{\left\< #1, #2 \right\>}
\newcommand{\set}[1]{\left\{ #1 \right\}}
\newcommand{\eps}{\varepsilon}
\newcommand{\tensor}{\otimes}
\newcommand{\Tensor}{\bigotimes}

\renewcommand{\P}{\mathbb{P}}

\newcommand{\prob}[2][]{\P_{#1}\left[ \, #2 \, \right] }

\DeclareMathOperator*{\argmin}{arg\,min}
\DeclareMathOperator*{\argmax}{arg\,max}

\newcommand{\wh}{\widehat}
\renewcommand\t{{\ensuremath{\scriptscriptstyle{\top}}}}

\newcommand\Parens[1]{\mleft(#1\mright)}

\newcommand\Braces[1]{\mleft\{#1\mright\}}

\newcommand\Abs[1]{\mleft|#1\mright|}

\newcommand\ind[1]{\ensuremath{\mathds{1}\{#1\}}}

\newcommand\qed{\vbox{\hrule height0.6pt\hbox{%
   \vrule height1.3ex width0.6pt\hskip0.8ex
   \vrule width0.6pt}\hrule height0.6pt}}

\newcommand\citep\cite
\newcommand\citet\cite

\numberwithin{equation}{section}

\title{Successive Rank-One Approximations FOR Nearly Orthogonally Decomposable Symmetric Tensors}
\author{Cun Mu\footnotemark[1]
  \and Daniel Hsu\footnotemark[2]
  \and Donald Goldfarb\footnotemark[1]}

\begin{document}
\maketitle
\slugger{simax}{xxxx}{xx}{x}{x--x}%slugger should be set to mms, siap, sicomp, sicon, sidma, sima, simax, sinum, siopt, sisc, or sirev

{\renewcommand{\thefootnote}{\fnsymbol{footnote}}%
\footnotetext[1]{Department of Industrial Engineering and Operations Research, Columbia University (\url{cm3052@columbia.edu}, \url{goldfarb@columbia.edu})}
\footnotetext[2]{Department of Computer Science, Columbia University, \url{djhsu@cs.columbia.edu}}
}

\begin{abstract}
  Many idealized problems in signal processing, machine learning and
  statistics can be reduced to the problem of finding the symmetric
  canonical decomposition of an underlying symmetric and orthogonally
  decomposable (SOD) tensor.
  Drawing inspiration from the matrix case, the successive rank-one
  approximations (SROA) scheme has been proposed and shown to yield
  this tensor decomposition exactly, and a plethora of numerical methods
  have thus been developed for the tensor rank-one approximation
  problem.
  In practice, however, the inevitable errors (say) from estimation,
  computation, and modeling necessitate that the input tensor can only be
  assumed to be a nearly SOD tensor---i.e., a symmetric tensor
  slightly perturbed from the underlying SOD tensor.
  This article shows that even in the presence of perturbation, SROA
  can still robustly recover the symmetric canonical decomposition of
  the underlying tensor.
  It is shown that when the perturbation error is small enough, the
  approximation errors do not accumulate with the iteration number.
  Numerical results are presented to support the theoretical findings.
\end{abstract}



\pagestyle{myheadings}
\thispagestyle{plain}
\markboth{C. MU, D. HSU, AND D. GOLDFARB}{RANK-1 APPROXIMATIONS FOR STRUCTURED TENSORS}

\section{Introduction}
The eigenvalue decomposition of symmetric matrices is one of the most
important discoveries in mathematics with an abundance of
applications across all disciplines of science and engineering. One
way to explain such a decomposition is to express the symmetric matrix
as a minimal sum of rank-one symmetric matrices. It is well known that
the eigenvalue decomposition can be simply obtained via {\em successive
rank-one approximation (SROA)}.
Specifically, for a symmetric matrix $\mb X \in \reals^{n \times n}$
with rank $r$, one approximates $\mb X$ by a rank-one matrix to
minimize the Frobenius norm error:
\begin{flalign}\label{eqn:matrix_rank_one}
(\lambda_1, \bm x_1) \in \argmin_{\lambda \in \reals, \norm{\bm x}{} =
1} \norm{\mb X - \lambda \bm x \bm x^\t}{F} ;
\end{flalign}
then, one approximates the residual $\mb X - \lambda_1 \bm x_1 \bm
x_1^\t$ by another rank-one matrix $\lambda_2\bm x_2 \bm x_2^\t$,
and so on.
The above procedure continues until  one has found $r$ rank-one
matrices $(\lambda_i \bm x_i \bm x_i^\t)_{i=1}^r$; their
summation, $\sum_{i=1}^r \lambda_i  \bm x_i \bm x_i^\t$, yields an
eigenvalue decomposition of $\bm X$.
Moreover, due to the optimal approximation property of the
eigendecomposition, for any positive integer $k \le r$, the best
rank-$k$ approximation (in the sense of either the Frobenius norm or the
operator norm) to $\mb X$ is simply given by $\sum_{i=1}^k
\lambda_i \bm x_i \bm x_i^\t$ \cite{eckart1936approx}.

In this article, we study decompositions of higher-order
symmetric tensors, a natural generalization of symmetric matrices.
Many applications in signal processing, machine learning, and
statistics, involve higher-order interactions in data; in these cases,
higher-order tensors formed from the data are the primary objects of
interest.
A tensor $\mcb T \in \Tensor_{i=1}^p \reals^{n_i} := \reals^{n_1\times
n_2 \times \dotsb \times n_p}$ of order $p$ is called symmetric if
$n_1 = n_2 = \dotsb = n_p = n$ and its entries are invariant under any
permutation of their indices.
Symmetric tensors of order two ($p=2$) are symmetric matrices.
In the sequel, we reserve the term \emph{tensor} (without any further
specifiction) for tensors of order $p\ge3$.
A symmetric rank-one tensor can be naturally defined as a $p$-fold
outer product
\[
  \bm v^{\tensor p} := \underbrace{\bm v \tensor \bm v \tensor \cdots
  \tensor \bm v}_{\text{$p$ times}} ,
\]
where $\bm v \in \reals^n$ and $\tensor$ denotes the outer product
between vectors.\footnote{%
  For any $1\le i_1,i_2, \dotsc, i_p\le n$ and any $\bm v \in \bb
  R^n$, the $(i_1,i_2,\dotsc,i_p)$-th entry of $\bm v^{\tensor p}$ is
  $(\bm v^{\tensor p})_{i_1, i_2, \dotsc, i_p} = v_{i_1} v_{i_2} \dotsb
  v_{i_p}$.
}
The minimal number of rank-one symmetric tensors whose
sum is $\mcb T$ is called the {\em symmetric tensor rank} in the literature,
and any corresponding decomposition is called a {\em symmetric canonical
decomposition} \cite{harshman70parafac}.
Such decompositions have applications in many scientific and
engineering
domains~\cite{McCullagh87tensor,comon1994independent, smilde2004chemical, kolda2009tensor, comon2010handbook,JMLR:v15:anandkumar14b}.

%higher-order statistics \cite{McCullagh87tensor}, independent component analysis \cite{comon1994independent, comon2010handbook}, and very recently, parameter estimations in latent variable models \cite{anandkumar2012tensor}.

By analogy to the matrix case, successive rank-one approximations
schemes have been proposed for symmetric tensor decomposition \cite{zhang2001rank, kofidis2002best, kolda2005web, wang2007successive}.
Just as in the matrix case, one first approximates $\mcb T$ by a
symmetric rank-one tensor
\begin{flalign}\label{eqn:tensor_rank_one}
  (\lambda_1, \bm v_1) \in \argmin_{\lambda \in \reals, \norm{\bm v}{}
  = 1} \norm{\mcb T - \lambda \bm v^{\tensor p}}{F} ,
\end{flalign}
and then approximate the residual $\mcb T - \lambda_1 \bm v_1^{\tensor p}$
again by another symmetric rank-one tensor, and so on.
(The Frobenius norm $\norm{\cdot}{F}$ for tensors is defined later,
but is completely analogous to the matrix case.)
This process continues until a certain stopping criterion is met.
However, unlike symmetric matrices, the above procedure for higher
order tensors ($p \ge 3$) faces a number of {\em computational} and {\em
theoretical} challenges.

Unlike problem \eqref{eqn:matrix_rank_one}---which can be
solved efficiently using simple techniques such as power iterations---
solving the rank-one approximation to higher order tensors is much
more difficult: it is NP-hard, even for symmetric third-order
tensors~\cite{hillar2009most}.
Researchers in numerical linear algebra and numerical optimization
have devoted a great amount of effort to solve problem
\eqref{eqn:tensor_rank_one}.
Broadly speaking, existing methods for problem
\eqref{eqn:tensor_rank_one} can be categorized into three types.
First, as problem \eqref{eqn:tensor_rank_one} is equivalent to finding
the extreme value of a homogeneous polynomial over the unit sphere,
general-purpose polynomial solvers based on the Sum-of-Squares (SOS) framework
\cite{shor87poly,nesterov00squared,parrilo2000structured,lasserre2001global,parrilo2003semidefinite}, such as GloptiPoly 3
\cite{henrion2009gloptipoly} and SOSTOOLS \cite{sostools}, can be
effectively applied to the rank-one approximation problem.
The SOS approach can solve any polynomial problem to any given
accuracy through a sequence of semidefinite programs; however, the
size of these programs are very large for high-dimensional problems,
and hence these techniques are generally limited to relatively small-sized
problems.
The second approach is to treat problem \eqref{eqn:tensor_rank_one} as
a nonlinear program \cite{bertsekas1999nonlinear,
wright1999numerical}, and then to exploit and adapt the wealth of
ideas from numerical optimization.
The resulting methods---which include \cite{de2000best, zhang2001rank,
kofidis2002best, wang2007successive, kolda2011shifted,
han2012unconstrained, chen2012maximum, zhang2012best, l2014sequential}
to just name a few---are empirically efficient and scalable, but are
only guaranteed to reach a local optimum or stationary point of the
objective over the sphere.
Therefore, to maximize their performance, these methods need to run
with several starting points.
The final approach is based on the recent trend of relaxing seemingly
intractable optimzation problems such as
problem~\eqref{eqn:tensor_rank_one} with more tractable convex
optimization problems that can be efficiently
solved~\cite{jiang2012tensor, nie2013semidefinite,
yang2014properties}.
The tensor structure in~\eqref{eqn:tensor_rank_one} has made it
possible to design highly-tailored convex relaxations that appear to
be very effective.
For instance, the semidefinite relaxation approach in ~\cite{jiang2012tensor} was able to globally solve almost all the
randomly generated instances that they tested. {Aside from the above solvers, a few algorithms have been specifically designed for the scenario where some side information regarding the solution of \eqref{eqn:tensor_rank_one} is known. For example, when the signs of the optimizer $\bm v_1$ are revealed, polynomial time approximation schemes for solving \eqref{eqn:tensor_rank_one} are available \cite{ling2009biquadratic}.}

In contrast to the many active efforts and promising results on the computational side, the
theoretical properties of successive rank-one approximations are far
less developed.
Although SROA is justified for matrix eigenvalue decomposition, it is
known to fail for general tensors \cite{stegeman2010subtracting}.
Indeed, much has been established about the failings of low-rank
approximation concepts for tensors that are taken for granted in the
matrix
setting~\cite{kolda2001orthogonal,kolda2003counter,
stegeman2007degeneracy,stegeman2008low,silva2008illposed}.
For instance, the best rank-$r$ approximation to a general tensor is not even guaranteed to exist (though several
sufficient conditions for this existence have been recently proposed \cite{lim2010multiarray, lim2014blind}).
Nevertheless, SROA can be still justified for certain classes of
symmetric tensors that arise in applications.

\paragraph{Nearly SOD tensors} Indeed, in many applications (e.g., higher-order statistical
estimation~\cite{McCullagh87tensor}, independent component
analysis~\cite{comon1994independent, comon2010handbook}, and parameter
estimation for latent variable models \cite{JMLR:v15:anandkumar14b}),
the input tensor $\wh{\mcb T}$ may be fairly assumed to be a symmetric
tensor slightly perturbed from {\em a symmetric and orthogonally
decomposable (SOD) tensor}  $\mcb T$.
That is,
\[
  \wh{\mcb T} = \mcb T + \mcb E ,
\]
where the underlying SOD tensor may be written as $\mcb T =
\sum_{i=1}^r \lambda_i \bm v_i^{\tensor p}$ with $\innerprod{\bm
v_i}{\bm v_j} = \ind{i=j}$ for $1 \leq i, j \leq r$, and $\mcb E$
is a perturbation tensor.
In these aforementioned applications, we are interested in obtaining
the underlying pairs $\{(\lambda_i,\bm v_i)\}_{i=1}^r$.
When $\mcb E$ vanishes, it is known that $\sum_{i=1}^r \lambda_i \bm
v_i^{\tensor p}$ is the unique symmetric canonical
decomposition~\cite{harshman70parafac,kruskal77threeway}, and
moreover, successive rank-one approximation exactly recovers
$\{(\lambda_i, \bm v_i)\}_{i=1}^r$~\cite{zhang2001rank}.
However, because of the inevitable perturbation term $\mcb E$ arising
from sampling errors, noisy measurements, model misspecification,
numerical errors, and so on, it is crucial to understand {\em the behavior of SROA when $\mcb E \neq \mb 0$}.
In particular, one may ask if SROA provides an accurate approximation
to $\{(\lambda_i, \bm v_i)\}_{i=1}^r$.
If the answer is affirmative, then we can indeed take advantage of those
sophisticated numerical approaches to
solving~\eqref{eqn:tensor_rank_one} mentioned above for many practical
problems.
This is the focus of the present paper.

\paragraph{Algorithm-independent analysis} The recent work of~\cite{JMLR:v15:anandkumar14b} proposes a randomized
algorithm for approximating SROA based on the power method
of~\cite{de2000best}.
There, an error analysis specific to the proposed randomized algorithm (for the case $p=3$) shows that the decomposition
$\{(\lambda_i, \bm v_i)\}_{i=1}^r$ of $\mcb T$ can be approximately
recovered from $\wh{\mcb T}$ in polynomial time with high
probability---provided that the perturbation $\mcb E$ is sufficiently
small (roughly on the order of $1/n$ under a natural measure).
Our present aim is to provide a general analysis that is {\em independent}
of the specific approach used to obtain rank-one approximations and it seems to be beneficial.
Our analysis shows that the general SROA scheme in fact allows for
perturbations to be of the order $1/\sqrt[p-1]{n}$, suggesting advantages of using more sophisticated optimization
procedures and potentially more computational resources to solve each rank-one approximation step.

As motivation, we describe a simple and typical application from probabilistic
modeling where perturbations of SOD tensors naturally arise.
But first, we establish notations used throughout the paper, largely
borrowed from~\cite{lim2005eigenvalue}.

\paragraph{Notation}
A real $p$-th order $n$-dimensional tensor
$\mcb A \in \Tensor^p \reals^n := \reals^{n\times n \times \dotsb \times n}$,
\[
\mcb A = \left(\mc A_{i_1, i_2, \dotsc, i_p} \right), \;\; \mc A_{i_1,
i_2, \dotsc, i_p} \in \reals, \quad 1\le i_1, i_2, \dotsc, i_p \le n,
\]
is called \emph{symmetric} if its entries are invariant under any
permutation of their indices: for any $i_1, i_2, \dotsc, i_p \in [n]
:= \set{1,2,\dotsc,n}$,
\[
  \mc A_{i_1, i_2, \dotsc, i_p} = \mc A_{i_{\pi(1)}, i_{\pi(2)},
  \dotsc, i_{\pi(p)}}
\]
for any permutation mapping $\pi$ on $[p]$.

In addition to being considered as a multi-way array, a
tensor $\mcb A \in \Tensor^p \reals^n$ can also be interpreted as a
multilinear map in the following sense: for any matrices $\mb V_i \in
\reals^{n \times m_i}$ for $i \in [p]$, we interpret
$\mcb A(\mb V_1, \mb V_2, \dotsc, \mb V_p)$ as a tensor in
$\reals^{m_1\times m_2 \times \cdots \times m_p}$ whose $(i_1, i_2,
\dotsc, i_p)$-th entry is
\[
\left(\mcb A(\mb V_1, \mb V_2, \dotsc, \mb V_p)\right)_{i_1, i_2,
\dotsc, i_p} := \sum_{j_1, j_2, \dotsc, j_p \in [n]} \mcb A_{j_1, j_2,
\dotsc, j_p} (\mb V_1)_{j_1 i_1} (\mb V_2)_{j_2 i_2} \cdots  (\mb V_p)_{j_p i_p}.
\]
The following are special cases:
\begin{enumerate}
  \item $p=2$ (i.e., $\mcb A$ is an $n\times n$ symmetric matrix):
    \[
      \mcb A(\mb V_1, \mb V_2) = \mb V_1^\t \mcb A \mb V_2 \in
      \reals^{m_1 \times m_2}.
    \]

     \item For any $i_1, i_2, \dotsc, i_p \in [n]$,
    \[
      \mcb A(\bm e_{i_1}, \bm e_{i_2}, \dotsc, \bm e_{i_p}) = \mc A_{i_1, i_2,
      \dotsc, i_p},
    \]
    where $\bm e_i$ denotes the $i$-th standard basis of $\reals^n$ for any $i\in [n]$.

  \item $\mb V_i = \bm x \in \reals^n$ for all $i \in [p]$:
    \[
      \mcb A \bm x^{\tensor p} := \mcb A(\bm x, \bm x, \dotsc, \bm x)=
      \sum_{i_1,i_2, \dotsc, i_p \in [n]} \mc A_{i_1, i_2, \dotsc, i_p} \,
      x_{i_1} x_{i_2} \dotsb x_{i_p},
    \]
    which defines a  homogeneous polynomial of degree $p$.

  \item $\mb V_i = \bm x \in \reals^n$ for all $i \in [p-1]$, and $\mb
    V_p = \mb I \in \reals^{n \times n}$:
    \begin{align*}
      \mcb A \bm x^{\tensor p-1}
      & :=
      \mcb A(\bm x,\dotsc,\bm x, \mb I)
      \in \reals^n, \\
      \left(\mcb A \bm x^{\tensor p-1}\right)_i
      & \hphantom:= \sum_{i_1,i_2, \dotsc, i_{p-1} \in [n]}
      \mc A_{i_1, i_2, \dotsc, i_{p-1},i} \, x_{i_1} x_{i_2} \dotsb
      x_{i_{p-1}}
      .
    \end{align*}

\end{enumerate}

For any tensors $\mcb A$, $\mcb B \in \Tensor^p \reals^n$, the inner
product between them is defined as
\[
  \innerprod{\mcb A}{\mcb B} := \sum_{i_1,i_2, \dotsc,i_p \in [n]}
  \mc A_{i_1, i_2, \dotsc, i_p}  \mc B_{i_1, i_2, \dotsc, i_p}.
\]
Hence, $\mcb A \bm x^{\tensor p}$ (defined above) can also be interpreted as
the inner product between $\mcb A$ and $\bm x^{\tensor p}$.

Two tensor norms will be used in the paper.
For a tensor $\mcb A \in \Tensor^p \reals^n$, its \emph{Frobenius
norm} is $\norm{\mcb A}{F} := \sqrt{\innerprod{\mcb A}{\mcb A}}$, and
its \emph{operator norm} $\mcb A$, $\norm{\mcb A}{}$, is defined as
$\max_{\norm{\bm x_i}{} = 1} \mcb A(\bm x_1, \bm x_2, \dotsc, \bm
x_p)$.
It is well-known that for symmetric tensors $\mcb A$,
$\norm{\mcb A}{} = \max_{\norm{\bm x}{}=1} |\mcb A
\bm x^{\tensor p}|$ (see, e.g.,~\cite{chen2012maximum,
zhang2012best}).

\paragraph{A motivating example.}
To illustrate why we are particularly interested in nearly SOD
tensors, we now consider the following simple probabilistic model for
characterizing the topics of text documents.
(We follow the description from~\cite{JMLR:v15:anandkumar14b}.)
Let $n$ be the number of distinct topics in the corpus, $d$ be the
number of distinct words in the vocabulary, and $t \ge p$ be the
number of words in each document.
We identify the sets of distinct topics and words, respectively, by
$[n]$ and $[d]$.
The topic model posits the following generative process for a
document.
The document's topic $h \in [n]$ is first randomly drawn according to
a discrete probability distribution specified by $\bm w = (w_1, w_2,
\dotsc, w_n)$ (where we assume $w_i >0$ for each $i\in[n]$ and
$\sum_{i\in[n]} w_i = 1$):
\[
  \prob{h = i} = w_i \quad \text{for all $i \in [n]$}.
\]
Given the topic $h$, each of the document's $t$ words is then drawn
independently from the vocabulary according to the discrete
distribution specified by the probability vector $\bm \mu_h \in
\reals^d$; we assume that the probability vectors $\{\bm
\mu_i\}_{i\in[n]}$ are linearly independent (and, in particular, $d
\geq n$).
The task here is to estimate these probability vectors $\bm w$ and
$\{\bm \mu_i\}_{i\in[n]}$ based on a corpus of documents.

Denote by $\mb M_2 \in \reals^{d \times d}$ and $\mcb M_p \in \Tensor^p \reals^d$, respectively,  the pairs probability matrix and
$p$-tuples probability tensor, defined as follows: for all $i_1, i_2, \dotsc, i_p \in [d]$, ,
\begin{flalign*}
(\mb M_2)_{i_1,i_2} &= \prob{\text{1st word} = i_1,\; \text{2nd word}=i_2} \\
(\mcb M_p)_{i_1,i_2,\dotsc ,i_p} &= \prob{\text{1st word} = i_1,\; \text{2nd word}=i_2, \cdots,
\; \text{$p$th word}=i_p}.
\end{flalign*}
It can be shown that $\mb M_2$ and $\mcb M_p$ can be precisely
represented using $\bm w$ and $\{\bm \mu_i\}_{i\in[n]}$:
\begin{flalign*}
  \mb M_2 = \sum_{i\in[n]} w_i \bm \mu_i  \bm \mu_i^\t \quad \text{and} \quad
  \mcb M_p = \sum_{i\in[n]} w_i \bm \mu_i^{\tensor p}.
\end{flalign*}

Since $\mb M_2$ is positive semidefinite and $\rank(\mb M_2) = n$, $\mb
M_2 = \mb U \mb D \mb U^\t$ is its reduced eigenvalue
decomposition.
Here, $\bm U \in \reals^{d \times n}$ satisfies $\mb
U^\t \mb U = \mb I$, and $\mb D \in \reals^{n \times n}$ is a diagonal
matrix with $\diag{\mb D} \succ \bm 0$.
Now define
\[
  \mb W := \mb U \mb D^{-1/2} , \quad
  \lambda_i := w_i^{1-p/2}, \quad \text{and} \quad
  \bm v_i := \sqrt{w_i} \mb W^\t \bm \mu_i \in \reals^n \quad
  \text{for each $i \in [n]$} .
\]
Then
\[
  \mb W^\t \mb M_2 \mb W
  %= \mb D^{-1/2} \mb U^\t \mb U \mb D \mb U^\t \mb U \mb D^{-1/2}
  = \mb I
  = \mb M_2(\mb W, \mb W)
  = \sum_{i\in[n]} w_i (\mb W^\t \bm \mu_i)(\mb W^\t \bm \mu_i)^\t
  = \sum_{i\in[n]} \bm v_i \bm v_i^\t,
\]
which implies that $\{\bm v_1, \bm v_2, \dotsc, \bm v_n\}$ are orthogonal.
Moreover,
\begin{equation}
  \mcb T
  := \mcb M_p(\mb W, \mb W, \dotsc,\mb W)
  = \sum_{i\in[n]} w_i \, (\mb W^\t \bm \mu_i)^{\tensor p}
  = \sum_{i\in[n]} \lambda_i \bm v_i^{\tensor p}.
  \label{eqn:topic-decomp}
\end{equation}
Therefore, we can obtain $\{(\lambda_i, \bm v_i)\}_{i \in [n]}$ (and subsequently
$\{ (w_i,\bm \mu_i) \}_{i\in[n]}$ \footnote{%
  After obtaining $\{(\lambda_i, \bm v_i)\}_{i\in[n]}$, it is possible
  to obtain $\set{(w_i, \bm \mu_i)}_{i \in [n]}$ because for each $i
  \in [n]$, there exists $j \in [n]$ such that $w_i = \lambda_j^{2/(2-p)}$
  and $\bm \mu_i = \lambda_j (\mb W^\t)^\dag \bm v_j$, where $(\mb
  W^\t)^\dag$ denotes the Moore-Penrose pseudoinverse of $\bm W^\t$.%
}) by computing the (unique) symmetric canonical decomposition of tensor
$\mcb T$, which can be perfectly achieved by SROA
\cite{zhang2001rank}.

In order to obtain the tensor $\mcb T$, we need
$\mb M_2$ and $\mcb M_p$, both of which can be estimated from a
collection of documents.
Due to their independence, all pairs (resp., $p$-tuples) of words in a
document can be used in forming estimates of $\mb M_2$ (resp., $\mcb
M_p$).
However, the quantities $\mb M_2$ and $\mcb M_p$ are only known up
to sampling errors, and hence,  we are only able to construct a symmetric tensor
$\wh{\mcb T}$ that is, at best, only close to the one
in~\eqref{eqn:topic-decomp}.
A critical question is whether we can still use SROA (Algorithm
\ref{alg:rank-1}) to obtain an approximate decomposition and robustly
estimate the model parameters.

\begin{algorithm}
\caption{Successive Rank-One Approximation (SROA)}
\label{alg:rank-1}
\begin{algorithmic}[1]
  \renewcommand\algorithmicrequire{\textbf{input}}

  \REQUIRE symmetric tensor $\wh{\mcb T} \in \bigotimes^p \reals^n$.

  \STATE \textbf{initialize} $\wh{\mcb T}_0 := \wh{\mcb T}$
  \FOR{$i=1$ to $n$}
    \STATE $(\hat\lambda_i,\hat{\bm v}_i) \in
    \argmin_{\lambda\in\reals,\norm{\bm v}{}=1} \; \norm{\wh{\mcb
    T}_{i-1} - \lambda \bm v^{\tensor p}}{F}$.
    \label{line:rank_one_approx}

    \STATE $\wh{\mcb T}_i := \wh{\mcb T}_{i-1} - \hat{\lambda}_i
    \hat{\bm v}_i^{\tensor p}$.

  \ENDFOR

  \RETURN $\{ (\hat\lambda_i, \hat{\bm v}_i) \}_{i \in [n]}$.
\end{algorithmic}
\end{algorithm}

\paragraph{Setting}
Following the notation in the above example, in
the sequel, we denote $\wh{\mcb T} = \mcb T + \mcb E \in \Tensor^p
\reals^n$. Here $\mcb T$ is a symmetric tensor that is orthogonally
decomposable, i.e., $\mcb T = \sum_{i=1}^n \lambda_i \bm v_i^{\tensor p}$
with all $\lambda_i\neq 0$, $\set{\bm v_1, \bm v_2, \dotsc, \bm v_n}$ forming an orthonormal
basis of $\reals^n$, and $\mcb E$ is a symmetric perturbation tensor
with operator norm $\eps:=\norm{\mcb E}{}$. Note that in some
applications, we might instead have $\mcb T = \sum_{i=1}^r \lambda_i
\bm v_i^{\tensor p}$ for some $r < n$. Our results nevertheless can be
applied in that setting as well with little modification.

For simplicity, we also assume $p$ is odd and treat it as a constant
in big-$O$ notations. (We discuss the even case in
Section~\ref{sec:even}).
Without loss of generality, we can assume $\lambda_i > 0$ for all $i
\in [n]$, as we can always change the sign of $\bm v_i$ to make it
hold.
Moreover, line \ref{line:rank_one_approx} in
Algorithm~\ref{alg:rank-1} simply becomes
\begin{flalign*}
  \hat{\bm v}_i \in \argmax_{\norm{\bm v}{}=1} \wh{\mcb T}_{i-1} \bm
  v^{\tensor p} , \quad
  \hat \lambda_i = \wh{\mcb T}_{i-1} \hat{\bm v}_i^{\tensor p} .
\end{flalign*}

\paragraph{Organization}
Section \ref{sec:first_iter} analyzes the first iteration of
Algorithm \ref{alg:rank-1} and proves that $(\hat{\lambda}_1, \hat{\bm
v}_1)$ is a robust estimate of a pair $(\lambda_i, \bm v_i)$ for some
$i\in [n]$.
A full decomposition analysis is provided in Section
\ref{sec:full_deflation}, in which we establish the following property
of tensors: when $\norm{\mcb E}{}$ is small enough, the approximation
errors do not accumulate as the iteration number grows; in contrast,
the use of deflation is generally not advised in the matrix setting
for finding more than a handful of matrix eigenvectors due to
potential instability. Numerical experiments are also reported
to confirm our theoretical results.
\section{Rank-One Approximation} \label{sec:first_iter}
In this section, we provide an analysis of the first step of SROA
(Algorithm~\ref{alg:rank-1}), which yield a rank-one
approximation to $\wh{\mcb T}$.

\subsection{Review of Matrix Perturbation Analysis}

We first state a well-known result about perturbations of the
eigenvalue decomposition for symmetric matrices; this result serves as
a point of comparison for our study of higher-order tensors.
The result is stated just for rank-one approximations, and in a form
analogous to what we are able to show for the tensor case
(Theorem~\ref{thm:one_step} below).
\begin{theorem}[\cite{weyl1912asymptotische,davis1970rotation}]\label{thm:matrix_pert}
  Let $\mb M \in \reals^{n \times n}$ be a symmetric matrix with
  eigenvalue decomposition $\sum_{i=1}^n \lambda_i \bm v_i \bm
  v_i^\t$, where $|\lambda_1| \ge |\lambda_2| \ge \dotsb \ge
  |\lambda_n|$ and $\{ \bm v_1, \bm v_2, \dotsc, \bm v_n \}$ are
  orthonormal.
  Let $\wh{\mb M} = \mb M + \mb E \in \reals^{n \times n}$ for some
  symmetric matrix $\mb E$ with $\eps := \norm{\mb E}{}$, and let
  \[
    (\hat\lambda,\hat{\mb x}) \in
    \argmin_{\lambda\in\reals,\norm{\bm x}{}=1}
    \norm{\wh{\mb M}-\lambda \bm x \bm x^\t}{F} .
  \]
  The following holds.
  \begin{itemize}
    \item (Perturbation of leading eigenvalue.)
      $|\hat{\lambda} - \lambda_1| \le \eps$.

    \item (Perturbation of leading eigenvector.)
      Suppose $\gamma := \min_{i\neq1}|\lambda_1-\lambda_i| > 0$.
      Then $\innerprod{\hat{\bm x}}{\bm v_1}^2 \geq 1 -
      (2\eps/\gamma)^2$.
      This implies that if $2\eps/\gamma \leq 1$, then
      $\min\{\norm{\bm v_1 - \hat{\bm x}}{},\, \norm{\bm v_1 +
      \hat{\bm x}}{}\} \leq O(\eps/\gamma)$.

  \end{itemize}

\end{theorem}
For completeness, we give a proof of the eigenvector perturbation
bound in Appendix~\ref{sec:matrix_pert} since it is not directly
implied by results in~\cite{davis1970rotation} but essentially uses
the same arguments.

\subsection{Single Rank-One Approximation}

The main result of this section concerns the first step of SROA
(Algorithm \ref{alg:rank-1}) and establishes a perturbation result for
nearly SOD tensors.
\begin{theorem}\label{thm:one_step}
  For any odd positive integer $p \geq 3$,
  let $\wh{\mcb T} := \mcb T + \mcb E \in \Tensor^p \reals^n$, where
  $\mcb T$ is a symmetric tensor with orthogonal decomposition $\mcb T
  = \sum_{i=1}^n \lambda_i \bm v_i^{\tensor p}$, $\set{\bm v_1, \bm
  v_2, \dotsc, \bm v_n}$ is an orthonormal basis of $\reals^n$,
  $\lambda_i > 0$ for all $i \in [n]$, and $\mcb E$ is a symmetric
  tensor with operator norm $\eps:=\norm{\mcb E}{}$.
  Let $\hat{\bm x} \in \arg \max_{\norm{\bm x}{2}=1} \wh{\mcb T}\bm
  x^{\otimes p}$ and $\hat{\lambda} := \wh{\mcb T}\hat{\bm x}^{\otimes
  p}$.
  Then there exists $j \in [n]$ such that
  \begin{flalign}\label{eqn:one_step}
    |\hat \lambda - \lambda_j| \le \eps,
    \quad
    \norm{\hat{\bm x}-\bm v_j}{2} \le 10\left(\frac{\eps}{\lambda_j} +
    \left(\frac{\eps}{\lambda_j}\right)^2 \right).
  \end{flalign}
\end{theorem}

To prove Theorem~\ref{thm:one_step}, we first establish an
intermediate result.
Let $x_i := \innerprod{\bm v_i}{\hat{\bm x}}$, so $\hat{\bm x} =
\sum_{i=1}^n x_i \bm v_i$ and $\sum_{i=1}^n x_i^2 = 1$ since $\{\bm v_1, \bm
v_2, \dotsc, \bm v_n\}$ is an orthonormal basis for $\reals^n$ and
$\norm{\hat{\bm x}}{}=1$.
We reorder the indicies $[n]$ so that
\begin{flalign}
  \lambda_1 |x_1|^{p-2} \ge \lambda_2 |x_2|^{p-2} \ge \dotsb \ge
  \lambda_n |x_n|^{p-2} .
\end{flalign}
Our intermediate result, derived by simply bounding $\hat \lambda$
from both above and below, is as follows.
\begin{lemma}\label{lemma:first_step}
  In the notation from above,
  \begin{flalign}\label{eqn:lemma}
    \lambda_1 \ge \lambda_{\max}-2 \eps,
    \quad |x_1| \ge 1 - 2\eps/\lambda_1,
    \quad x_1^2 \ge x_1^{p-1} \ge 1 -
    4\eps/\lambda_1,
    \quad \text{and} \quad |\hat{\lambda} - \lambda_1| \le \eps
    .
  \end{flalign}
where $\lambda_{\max} = \max_{i\in[n]} \lambda_i$.
\end{lemma}
\begin{proof}
To show \eqref{eqn:lemma}, we will bound $\hat{\lambda} = \wh{\mcb T}\hat{\bm x}^{\otimes p}$ from both above and below.

For the upper bound, we have
\begin{flalign}
  \hat \lambda = \wh{\mcb T}\hat{\bm x}^{\otimes p}
  & = \mcb T \hat{\bm x}^{\otimes p}+ \mcb E \hat{\bm x}^{\otimes p}
  \notag \\
  & = \sum_{i=1}^n \lambda_i x_i^p + \mcb E \hat{\bm x}^{\otimes p}
  \notag \\
  & \le \sum_{i=1}^n \lambda_i |x_i|^{p-2} x_i^2 + \eps \le \lambda_1
  |x_1|^{p-2} + \eps,
\label{eqn:upper_bd}
\end{flalign}
where the last inequality follows since $\sum_{i=1}^n x_i^2 = 1$.

On the other hand,
\begin{flalign}\label{eqn:lower_bd}
  \hat \lambda \ge \max_{i\in[n]} \mcb T {\bm v_i}^{\tensor p} - \norm{\mcb E}{} =  \lambda_{\max} - \eps \ge \lambda_1 - \eps.
\end{flalign}

Combining \eqref{eqn:upper_bd} and \eqref{eqn:lower_bd}, it can be easily verified that
\[
\lambda_1 \ge \lambda_{\max} - 2\eps, \quad |\lambda_1 - \hat{\lambda}| \le \eps.
\]
and moreover,
\[
|x_1| \ge |x_1|^{p-2} \ge \frac{\lambda_1 - 2 \eps}{\lambda_1} = 1 - \frac{2\eps}{\lambda_1}
\]
which implies that $x_1^{p-1} = |x_1|^{p-2} \cdot |x_1| \ge 1 - 4\eps/\lambda_1.$
\end{proof}
\begin{remark}
The higher-order requirement, $p\ge 3$, is crucial in the analysis. Specifically we can bound $|x_1|$ below by the lower bound of $|x_1|^{p-2}$, which can be done by bounding $\hat{\lambda} = \wh{\mcb T}\hat{\bm x}^{\otimes p}$ from both above and below. This essentially explains why Lemma \ref{lemma:first_step}, different from the matrix case ($p=2$), does not rely on the spectral gap condition.
\end{remark}

The bound $|\hat \lambda - \lambda_1| \le \eps$ proved in Lemma \ref{lemma:first_step} is comparable to the matrix counterpart in Theorem \ref{thm:matrix_pert}, and is optimal in the worst case. Consider ${\mcb T} = \sum_{i=1}^n \lambda_i \bm e_i^{\tensor p}$ with $\lambda_1 > \lambda_2 \ge \ldots \ge \lambda_n \ge 0$ and $\mcb E = \eps \bm e_1^{\tensor p}$, for some $\eps > 0$. Then clearly $\hat \lambda = \lambda_1+\eps$ and $|\hat \lambda - \lambda_1| = \eps$.

Moreover, when $\mcb E$ vanishes, Lemma \ref{lemma:first_step} leads
directly to the following result given in \cite{zhang2001rank}.
\begin{corollary}[\cite{zhang2001rank}]
  Suppose $\mcb E = \mb 0$ (i.e., $\wh{\mcb T} = \mcb T = \sum_{i=1}^n
  \lambda_i \bm v_i^{\tensor p}$ is orthogonally decomposable).
  Then Algorithm \ref{alg:rank-1} computes $\{(\lambda_i, \bm
  v_i)\}_{i=1}^n$ exactly.
\end{corollary}

However, compared to Theorem \ref{thm:matrix_pert}, the bound for $\mb
x$ is $|x_1| \ge 1- 2\eps/\lambda_1$ appears to be suboptimal; this is
because the bound only implies $\norm{\hat{\bm x}- \bm v_1}{} =
O(\sqrt{\eps/\lambda_1})$.
In the following, we will proceed to improve this result to
$\norm{\hat{\bm x}- \bm v_1}{} = O({\eps/\lambda_1})$ by using the
first-order optimality condition~\cite{wright1999numerical}.
See~\cite{lim2005eigenvalue} for a discussion in the present context.

Consider the Lagrangian function corresponding to the optimization
problem $\max_{\norm{\bm x}{}=1} \wh{\mcb T} \bm x^{\tensor p},$
\begin{flalign}
\mcb L(\bm x, \lambda):= \wh{\mcb T} \bm x^{\tensor p}-\frac{\lambda}{2}\left(\norm{\bm x}{}^2-1 \right),
\end{flalign}
where $\lambda \in \reals$ corresponds to the Lagrange multiplier for
the equality constraint.
As $\hat{\bm x} \in \arg \max_{\norm{\bm x}{2}=1} \wh{\mcb T}\bm
x^{\otimes p}$ (and the linear independent constraint qualification \cite{wright1999numerical}
can be easily verified), there exists $\bar \lambda \in \reals$ such
that
\begin{flalign}\label{eqn:1st_order_opt}
\nabla \mcb L(\hat{\bm x}, \bar \lambda) = \wh{\mcb T} \bm x^{\tensor p-1} - \bar \lambda \bm x= 0.
\end{flalign}
Moreover, as $\norm{\hat{\bm x}}{}=1$, $\bar \lambda = \bar \lambda
\innerprod{\bm x}{\bm x} = \wh{\mcb T} \bm x^{\tensor p}=\hat
\lambda$.
Thus we have $\hat \lambda \hat{\bm x} = \wh{\mcb T} \hat{\bm
x}^{\tensor p-1}$.

We are now ready to prove Theorem~\ref{thm:one_step}.
\begin{proof}[Proof of Theorem~\ref{thm:one_step}]
The first inequality in \eqref{eqn:one_step} has been proved in Lemma \ref{lemma:first_step}, so we are left to prove the second one.

From the first-order optimality condition above, we have
\[
\hat \lambda \hat{\bm x} = \wh{\mcb T} \hat{\bm x}^{\tensor p-1} = \mcb T \hat{\bm x}^{\tensor p-1} + \mcb E \hat{\bm x}^{\tensor p-1} = \lambda_1 x_1^{p-1}\bm v_1 + \sum_{i \ge 2} \lambda_i x_i^{p-1} \bm v_i + \mcb E \hat{\bm x}^{\tensor p-1}.
\]
Thus,
\begin{flalign}
\norm{\lambda_1 (\hat{\bm x} - \bm v_1)}{2} &= \norm{(\lambda_1 - \hat \lambda) \hat{\bm x} + (\hat \lambda \hat{\bm x} - \lambda_1 \bm v_1)}{2} \nonumber \\
& = \norm{(\lambda_1 - \hat \lambda) \hat{\bm x} + \lambda_1 (x_1^{p-1}-1)\bm v_1 + \sum_{i \ge 2} \lambda_i x_i^{p-1} \bm v_i + \mcb E \hat{\bm x}^{\tensor p-1}}{2} \nonumber \\
& \le |\lambda_1 - \hat \lambda| + \lambda_1 |x_1^{p-1}-1| + \norm{\sum_{i \ge 2} \lambda_i x_i^{p-1} \bm v_i }{2} + \norm{\mcb E \hat{\bm x}^{\tensor p-1}}{2} \label{eqn:thm_1_sum}
\end{flalign}
by the triangle inequality.
By Lemma \ref{lemma:first_step}, we have
\begin{flalign}\label{eqn:thm_1_bd1}
  |\lambda_1 - \hat{\lambda}| \le \eps, \quad |x_1^{p-1} -1| \le
  4\eps/\lambda_1, \quad  \text{and} \quad \norm{\mcb E \hat{\bm x}^{\tensor p-1}}{2} \le \eps.
\end{flalign}
Moreover,
\begin{flalign}\label{eqn:thm_1_bd2}
\norm{\sum_{i\ge 2} \lambda_i x_i^{p-1} \bm v_i}{2} = \left( \sum_{i\ge 2} \lambda_i^2 x_i^{2p-2} \right)^{1/2} &\le \lambda_2 |x_2|^{p-2} \sqrt{1-x_1^2} \nonumber\\
\le& \lambda_2(1-x_1^2) \le \frac{4\eps \lambda_2}{\lambda_1} \le 4\eps(1+2\eps/\lambda_1),
\end{flalign}
where we have used Lemma \ref{lemma:first_step} and the fact that $\lambda_2/\lambda_1 \le \lambda_{\max}/\lambda_1 \le 1+2\eps/\lambda_1.$
Substituting \eqref{eqn:thm_1_bd1} and \eqref{eqn:thm_1_bd2} back into \eqref{eqn:thm_1_sum}, we can easily obtain
\begin{flalign}\label{eqn:temp}
  \norm{\hat{\bm x}-\bm v_1}{2} \le 10\left(\frac{\eps}{\lambda_1} +
  \left(\frac{\eps}{\lambda_1}\right)^2 \right).
%  \qedhere
\end{flalign}
\end{proof}

\begin{remark}
When $p>3$, we can slightly sharpen \eqref{eqn:temp} to
\begin{flalign*}
  \norm{\hat{\bm x}-\bm v_1}{2} \le 8\frac{\eps}{\lambda_1} + 4\left(\frac{\eps}{\lambda_1}\right)^2,
%  \qedhere
\end{flalign*}
by replacing \eqref{eqn:thm_1_bd2} with
\[
\lambda_2(1-x_1^2) \le \lambda_2 (1-|x_1|^{p-2}) \le  \lambda_2 \cdot \frac{2\eps}{\lambda_1} \le 2 \eps(1+2\eps/\lambda_1).
\]

\end{remark}

Theorem~\ref{thm:one_step} indicates that the first step of SROA for a
nearly SOD tensor approximately recovers $(\lambda_j,\bm v_j)$ for
some $j \in [n]$.
In particular, whenever $\eps$ is small enough relative to $\lambda_1$
(e.g., $\eps \le \lambda_1/2$), there always exists $j\in [n]$ such
that $|\hat \lambda - \lambda_j| \le \eps$ and $\norm{\hat{\bm x} -\bm
v_j}{2} \le 10\cdot (1+1/2)\eps/\lambda_j = 15\eps/\lambda_j$.
This is analogous to Theorem \ref{thm:matrix_pert}, except that {\em the
spectral gap condition} required in Theorem \ref{thm:matrix_pert} is
{\em not necessary at all} for the perturbation bounds of SOD tensors.

\subsection{Numerical Verifications for Theorem \ref{thm:one_step}}
\begin{figure*}
  \begin{center}
    \setlength{\tabcolsep}{0.25in}
    \begin{tabular}{cc}
      \includegraphics[width=2in]{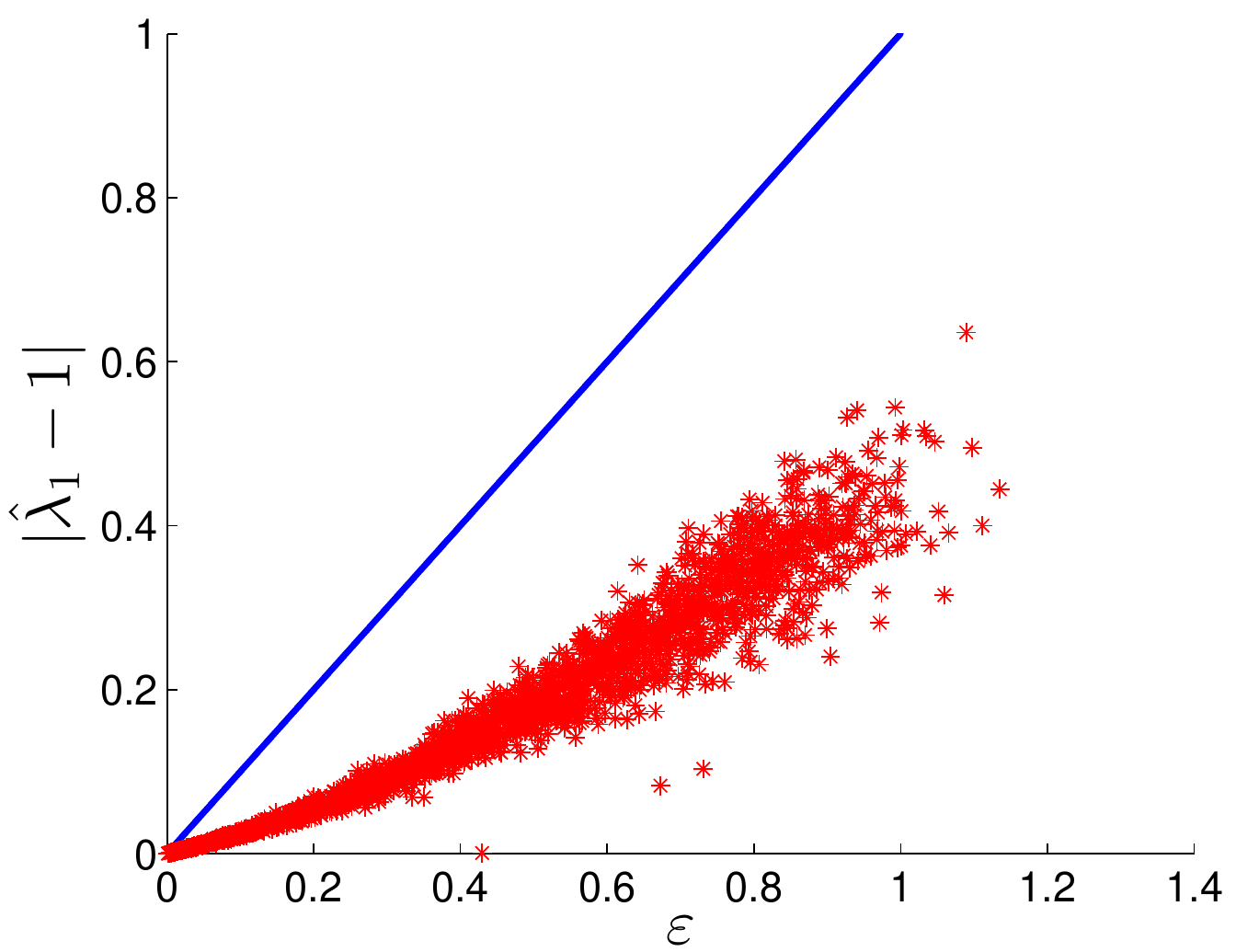} &
      \includegraphics[width=2in]{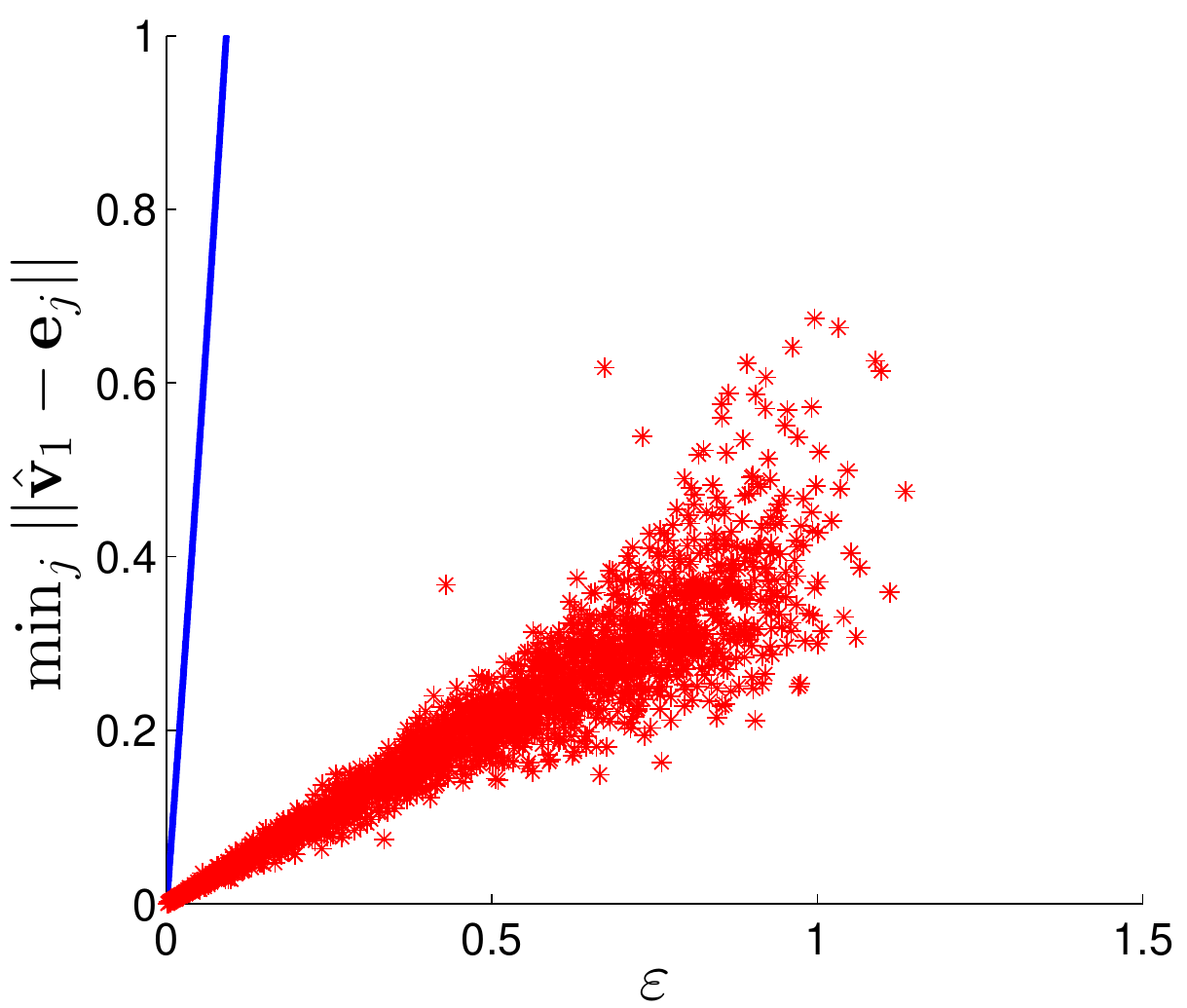}
      \\
      \multicolumn{2}{c}{\emph{Perturbation: binary}}
      \\
      \\
      \includegraphics[width=2in]{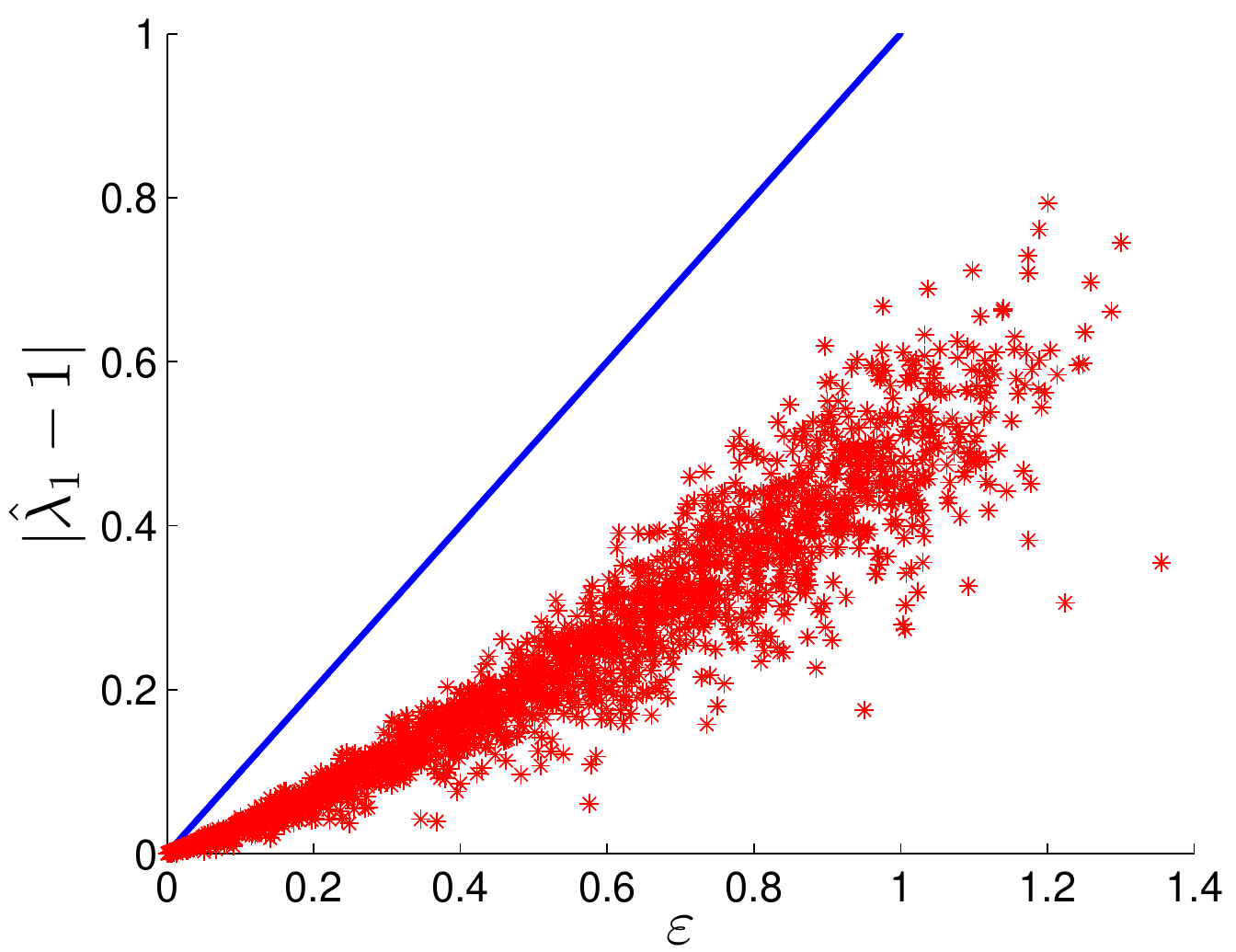} &
      \includegraphics[width=2in]{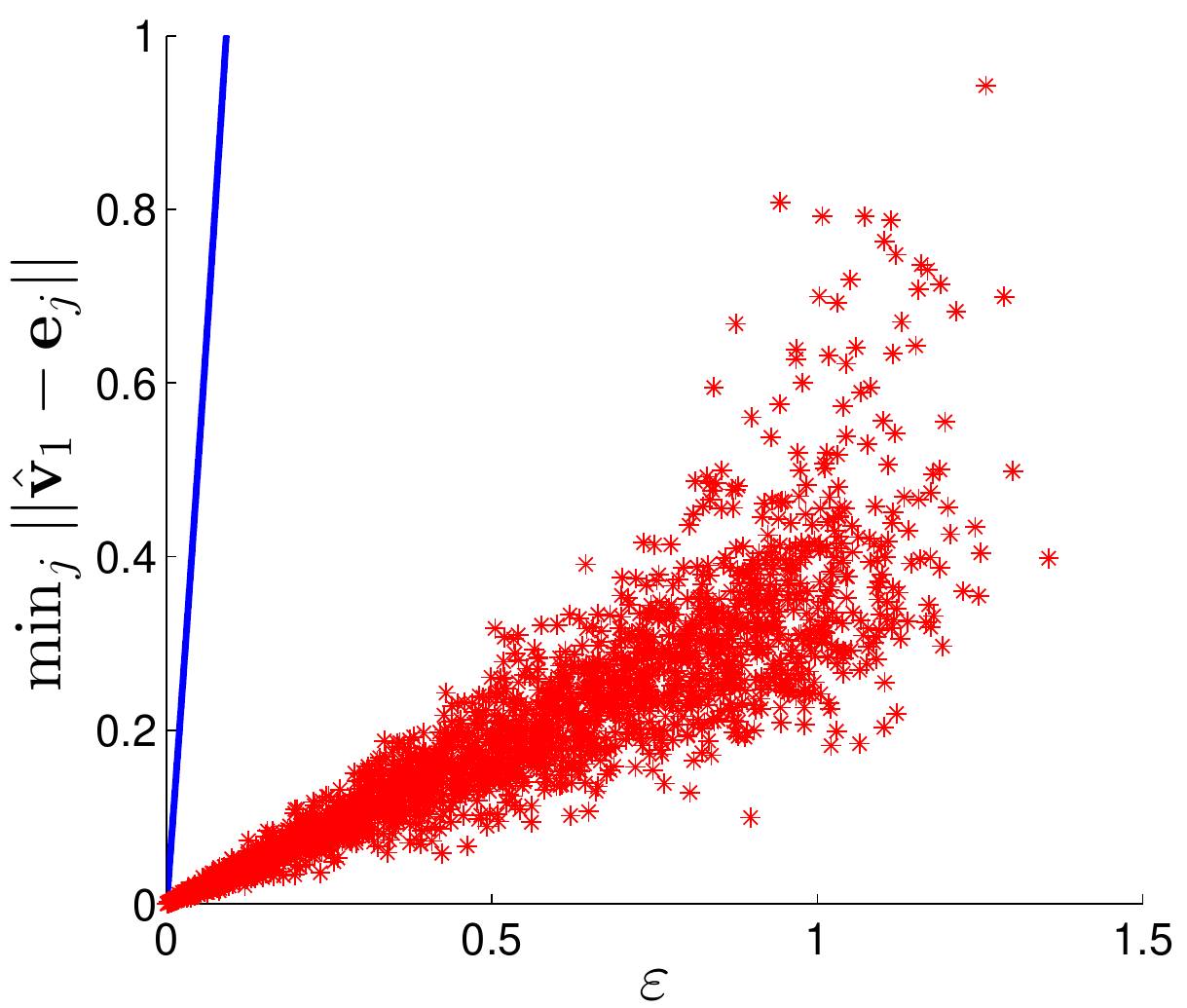}
      \\
      \multicolumn{2}{c}{\emph{Perturbation: uniform}}
      \\
      \\
      \includegraphics[width=2in]{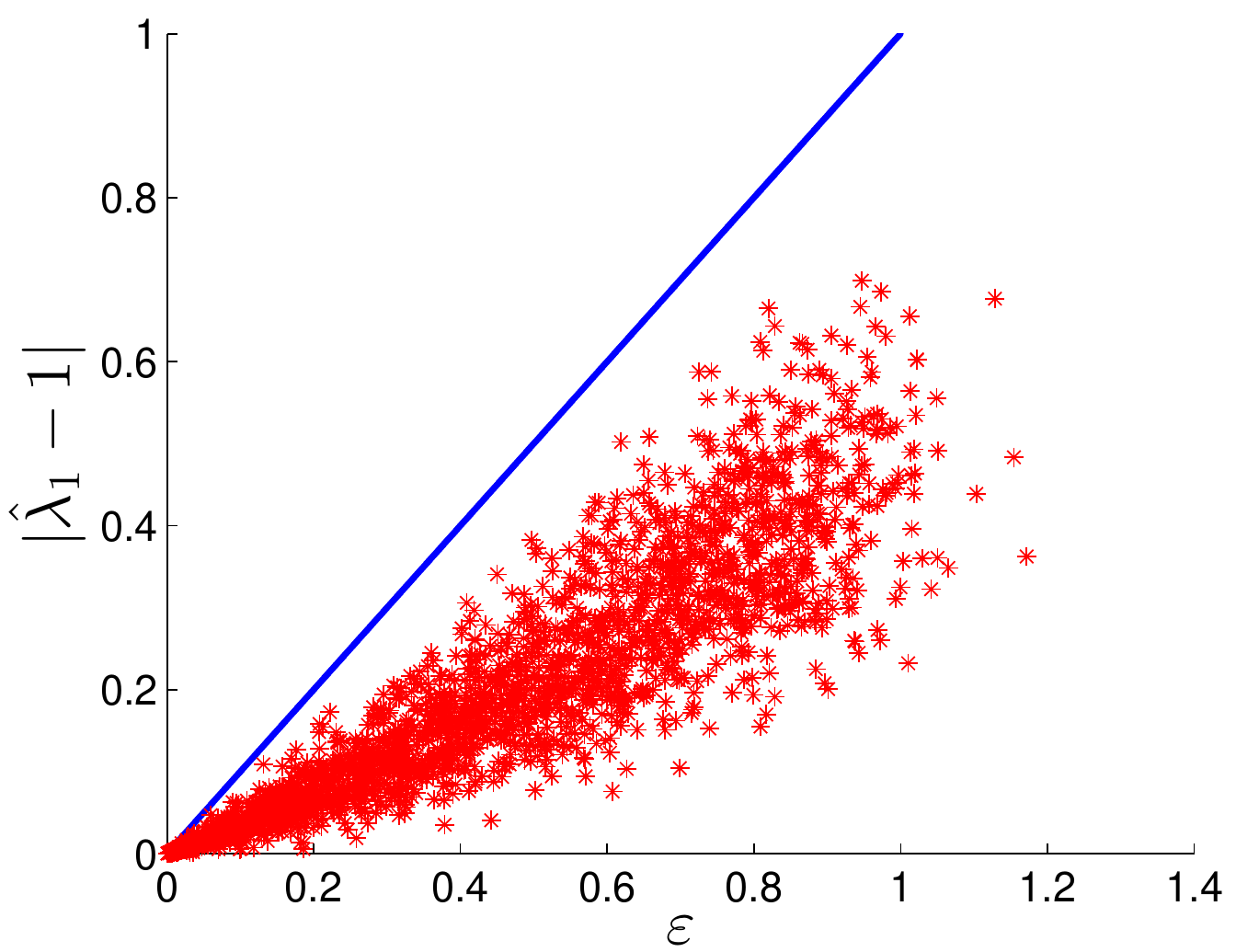} &
      \includegraphics[width=2in]{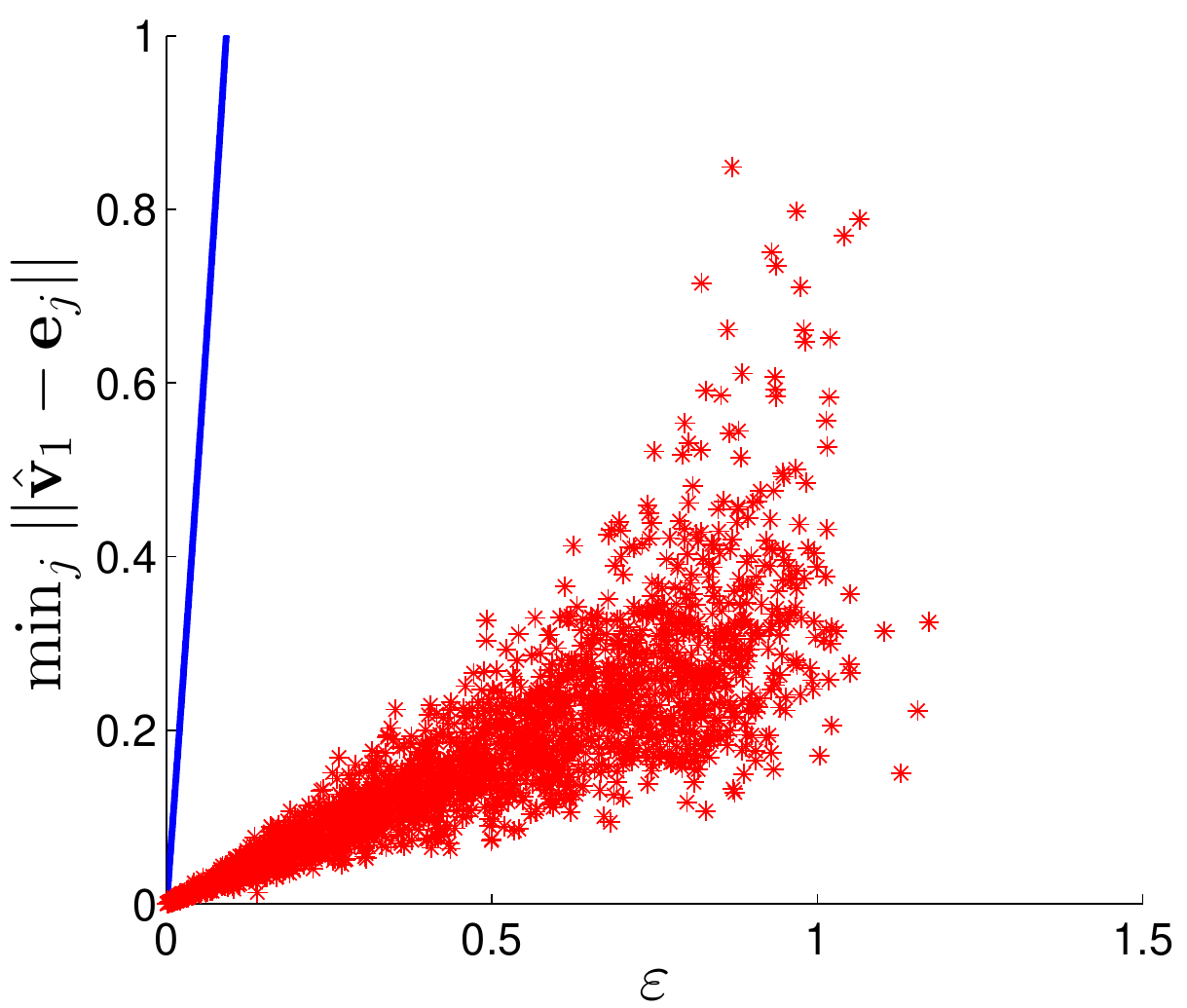}
      \\
      \multicolumn{2}{c}{\emph{Perturbation: Gaussian}}
      \\
      \\

    \end{tabular}
    \caption{%
      \textbf{Approximation Errors of the First Iteration.}
      The approximation errors in $\hat\lambda_1$ (resp., $\hat{\mb v}_1$)
      are plotted on the left (resp., right) as a function of the size
      of the perturbation $\eps$.
      Each (red) point corresponds to one randomly generated instance,
      and the (blue) solid line is the upper bound from
      Theorem~\ref{thm:one_step}.
    }
    \label{fig:first_iter}
  \end{center}
\end{figure*}

We generate nearly symmetric orthogonally decomposable tensors
$\wh{\mcb T} = \mcb T + \mcb E \in \reals^{10\times 10 \times 10}$ in
the following manner.
We let the underlying symmetric orthogonally decomposable tensor $\mcb
T$ be the diagonal tensor with all diagonal entries equal to one,
i.e., $\mcb T = \sum_{i=1}^{10} \bm e_i^{\tensor 3}$ (where $\bm e_i$
is the $i$-th coordinate basis vector).
The perturbation tensor $\mcb E$ is generated under the following three random
models:
\begin{description}
  \item[Binary:]
    independent entries $\mc E_{i,j,k} \in \{\pm\sigma\}$ uniformly at
    random;

  \item[Uniform:]
    indepedent entries $\mc E_{i,j,k} \in [-2\sigma,2\sigma]$
    uniformly at random;

  \item[Gaussian:]
    independent entries $\mc E_{i,j,k} \sim \mc N(0,\sigma^2)$;

\end{description}
where $\sigma$ is varied from $0.0001$ to
$0.2$ with increment $0.0001$, and one instance is generated for each value of $\sigma$.

For every randomly generated instance, we solve the polynomial
optimization problems
\begin{flalign}
\norm{\mcb E}{} &= \max_{\norm{\bm x}{} = 1} \mcb E \bm x^{\tensor 3}
  \quad \mbox{and}\quad \hat{\bm v}_1 \in \arg \max_{\norm{\bm x}{}=1}
  \wh{\mcb T} \bm x^{\tensor 3}
\end{flalign}
using the general polynomial solver
\textsf{GloptiPoly~3}~\citep{henrion2009gloptipoly}, and set $\hat
\lambda_1 := \wh{\mcb T} \hat{\bm v}_1^{\tensor 3}.$ 

In Figure \ref{fig:first_iter}, we plot the approximation error $|\hat
\lambda_1-1|$ and $\min_{j\in[10]} \norm{\hat{\bm v}_1- \bm e_j}{}$,
respectively against the value of norm $\norm{\mcb E}{}$.
Each (red) point corresponds to one randomly generated instance, and the
(blue) lines are the upper bounds given in Theorem \ref{thm:one_step}.
We observe no instance violating the theoretical bounds.

\section{Full Decomposition Analysis}\label{sec:full_deflation}
In the second iteration of Algorithm~\ref{alg:rank-1}, we have
\[
  \hat{\bm v}_2 \in \arg \max_{\norm{\bm x}{2}= 1} \wh{\mcb T}_1 \bm
  x^{\tensor p}, \quad \hat \lambda_2 = \wh{\mcb T}_1 \hat{\bm
  v}_2^{\tensor p}
  ,
\]
where, for some $j \in [n]$,
\[
  \wh{\mcb T}_1
  = \hat{\mcb T} - \hat \lambda_1 \hat{\bm v_1}^{\tensor p}
  = \sum_{i\neq j} \lambda_i \bm v_i^{\tensor p} + \wh{\mcb E}
  \quad\text{and}\quad
  \wh{\mcb E} = \mcb E + (\lambda_j\bm v_j^{\tensor p}- \hat \lambda_1
  \hat{\bm v_1}^{\tensor p})
  .
\]
%Recall that $\wh{\mcb T}_1 = \hat{\mcb T} - \hat \lambda_1 \hat{\bm
%v_1}^{\tensor p} = \sum_{i\neq j} \lambda_i \bm v_i^{\tensor p} +
%\hat{\mcb E}$, where $\wh{\mcb E} = \mcb E + (\lambda_j\bm
%v_j^{\tensor p}- \hat \lambda_1 \hat{\bm v_1}^{\tensor p})$ for some
%$j \in [n]$.
Theorem~\ref{thm:one_step} can be directly applied again by bounding
the error norm $\|\wh{\mcb E}\|$.
However, since
\begin{flalign*}
  \|\wh{\mcb E}\|
  & = \norm{\mcb E + (\lambda_j\bm v_j^{\tensor p}- \hat \lambda_1 \hat{\bm v_1}^{\tensor p})}{} \nonumber \\
  & = \norm{\mcb E + (\lambda_j - \hat \lambda_1) \bm v_j^{\tensor p} + \hat \lambda_1(\bm v_j^{\tensor p} - \hat{\bm v}_1^{\tensor p}) }{} \nonumber \\
  & \le \norm{\mcb E}{} + |\lambda_j - \hat \lambda_1| + \hat \lambda_1 \norm{\bm v_j^{\tensor p} - \hat{\bm v}_1^{\tensor p}}{} \nonumber \\
  & \le (2 + 10 \sqrt p) \eps + O(\eps^2/ \lambda_j) ,
\end{flalign*}
it appears that the approximation error may increase dramatically with
the iteration number.

Fortunately, a more careful analysis shows that approximation error
does not in fact accmulate in this way.
The high-level reason is that while the operator norm $\|\lambda_j \bm v_j^{\tensor p} -
\hat{\lambda}_1 \hat{\bm v}_1^{\tensor p}\|$ is of order $\eps$,
the relevant quantity is essentially $\left(\lambda_j \bm v_j^{\tensor p} -
\hat{\lambda}_1 \hat{\bm v}_1^{\tensor p}\right)$ operating on the direction of
$\hat{\bm v}_2$, i.e.
$|(\lambda_j \bm v_j^{\tensor p} -
\hat{\lambda}_1 \hat{\bm v}_1^{\tensor p}) \hat{\bm v}_2^{\tensor p}|$,
which only gives rise to a quantity of order $\eps^2$ because $p \geq 3$.
This enables us to keep the approximation errors under control.

The main result of this section is as follows.
\begin{theorem}\label{thm:main}
  Pick any odd positive integer $p \geq 3$.
  There exists a positive constant $c_0 = c_0(p) >0$ such that the
  following holds.
  Let $\wh{\mcb T} := \mcb T + \mcb E \in \Tensor^p \reals^n$, where
  $\mcb T$ is a symmetric tensor with orthogonal decomposition $\mcb T
  = \sum_{i=1}^n \lambda_i \bm v_i^{\tensor p}$, $\set{\bm v_1, \bm
  v_2, \dotsc, \bm v_n}$ is an orthonormal basis of $\reals^n$,
  $\lambda_i > 0$ for all $i \in [n]$, and $\mcb E$ is a symmetric
  tensor with operator norm $\eps:=\norm{\mcb E}{}$.
  Assume $\eps \le c_0 \lambda_{\min}/n^{1/(p-1)}$, where
  $\lambda_{\min} := \min_{i\in[n]} \lambda_i$.
  Let $\{ (\hat\lambda_i,\hat{\bm v}_i) \}_{i \in [n]}$ be the output
  of Algorithm~\ref{alg:rank-1} for input $\wh{\mcb T}$.
  Then there exists a permutation $\pi$ on $[n]$ such that
  \[
    |\lambda_{\pi(j)} - \hat{\lambda}_j| \le 2\eps, \qquad  \norm{\bm
    v_{\pi(j)}-\hat{\bm v}_j}{} \le 20\eps/\lambda_{\pi(j)}  , \quad
    \forall j \in [n]
    .
  \]
\end{theorem}

\subsection{Deflation Analysis}
The proof of Theorem~\ref{thm:main} is based on the following lemma,
which provides a careful analysis of the errors introduced in $\mcb
T_i$ from steps $1,2,\dotsc,i$ in Algorithm~\ref{alg:rank-1}.
This lemma is a generalization of a result
from~\cite{JMLR:v15:anandkumar14b} (which only dealt with the $p=3$
case) and also more transparently reveals the sources of errors that result from
deflation.

\begin{lemma}\label{lem:main}
  Fix a subset $S \subseteq [n]$ and assume that $0 \leq \hat\eps \leq
  \lambda_i/2$ for each $i\in S$.
  Choose any $\{(\hat\lambda_i,\hat{\bm v}_i)\}_{i \in S} \subset \reals
  \times \reals^n$ such that
  \[
    |\lambda_i - \hat \lambda_i| \le \hat \eps , \quad
    \|\hat{\bm v}_i\| = 1 , \quad \text{and} \quad
    \innerprod{\bm v_i}{\hat{\bm v}_i}\ge 1- 2(\hat \eps/\lambda_i)^2 > 0
    ,
  \]
  and define $\mb \Delta_i := \lambda_i \bm v_i^{\otimes p} -
  \hat{\lambda}_i \hat{\bm v}_i^{\otimes p}$ for $i \in S$.
  Pick any unit vector $\bm x = \sum_{i=1}^n x_i \bm v_i$.
  Let $S_1 \subseteq S$ be the indices $i\in[n]$ such that $\lambda_i
  |x_i| \ge 4 \hat \eps $, and let $S_2 := S \setminus S_1$.
  Then
  \begin{flalign}
    \norm{\sum_{i \in S_1} \mb \Delta_i \bm x^{\otimes p-1}}{2}
    & \le 2^{p+1}p \bigg(\sum_{i\in S_1} x_i^{2(p-2)}\bigg)^{1/2}\hat
    \eps + 2^{p+1} \sum_{i\in S_1} |x_i|^{p-1} \hat \eps ,
    \label{eqn:deflation_S1}
    \\
    \norm{\sum_{i \in S_2} \mb \Delta_i \bm x^{\otimes p-1}}{2}
    & \le 6^p \left( \sum_{i\in S_2} \left(\frac{\hat
    \eps}{\lambda_i}\right)^{2(p-2)} \right)^{1/2}  \hat \eps  + 6^p
    \sum_{i \in S_2} \left(\frac{\hat \eps}{\lambda_i}\right)^{p-1}
    \hat \eps
    .
    \label{eqn:deflation_S2}
  \end{flalign}
  These imply that there exists positive constants $c_1, c_2 > 0$,
  depending only on $p$, such that
  \begin{flalign}
    \norm{\sum_{i \in S_1} \mb \Delta_i \bm x^{\otimes p-1}}{2}
    & \le c_1 \cdot \left(\sum_{i \in S_1} |x_i|^{p-1} \hat \eps\right) ,
    \\
    \norm{\sum_{i \in S_2} \mb \Delta_i \bm x^{\otimes p-1}}{2}
    & \le c_2 \cdot \left(\sum_{i \in S_2} \left(\frac{\hat
    \eps}{\lambda_i}\right)^{p-1} \hat \eps\right) ,
    \label{eqn:lem_S2}
    \\
    \norm{\sum_{i \in S_{\hphantom1}} \mb \Delta_i \bm x^{\otimes p-1}}{2}
    & \le c_1 \cdot \left(\sum_{i \in S} |x_i|^{p-1} \hat \eps\right) +
    c_2 \cdot \left(|S| \left(\frac{\hat \eps}{\min_{i\in S}\lambda_i}\right)^{p-1}
    \hat \eps\right)
    \label{eqn:lem_S_all}
    .
  \end{flalign}
%We have
%\[
%\norm{\sum_{i \in S} \Delta_i \bm x^{\otimes p-1}}{2} \le C_3 \left(\sum_{i\in S} |x_i|^{p-1}  + k(\hat \eps/\lambda_{\min})^{p-1} \right) \hat \eps.
%\]
\end{lemma}
\begin{remark}
Lemma \ref{lem:main} indicates that the accumulating error $\sum_{i \in S} \mb \Delta_i$ much less severely affects vectors that are incoherent with $\set{\bm v_i: i\in S}$. For instance, $\norm{\sum_{i \in S} \mb \Delta_i \bm v_i^{\otimes p-1}}{} = O(\hat \eps^2)$ for $i \in [n]\setminus S$, while $\norm{\sum_{i \in S} \mb \Delta_i \bm v_i^{\otimes p-1}}{} = O(\hat \eps)$ for $i \in S$.
\end{remark}
\subsection{Proof of Main Theorem}

We now use Lemma~\ref{lem:main} to prove the main theorem.

\begin{proof}[Proof of Theorem~\ref{thm:main}]
  It suffices to prove that the following property holds for each $i
  \in [n]$:
  \begin{equation}
    \text{there is a permutation $\pi$ on $[i]$ s.t., for all
    $j \in [i]$} ,
    \
    \begin{cases}
      |\lambda_{\pi(j)} - \hat{\lambda}_j|
      \le 2\eps , \ \text{and} \\
      \norm{\bm v_{\pi(j)}-\hat{\bm v}_j}{}
      \le \frac{20\eps}{\lambda_{\pi(j)}} .
    \end{cases}
%    \begin{aligned}
%    \end{aligned}
    \tag{$*$}
    \label{eqn:indhyp}
  \end{equation}
%{\em $(*)$ after the first $i$ runs of the loop in the Algorithm \ref{alg:rank-1}, there exists a permutation $\pi$ on $[k]$ such that
%\[
%|\lambda_{\pi(j)} - \hat{\lambda}_j| \le 2\eps, \qquad  \norm{\bm v_{\pi(j)}-\hat{\bm v}_j}{} \le 20\eps/\lambda_{\pi(j)}  , \quad \forall j \in [i].
%\]}
  The proof is by induction.
  The base case of~\eqref{eqn:indhyp} (where $i = 1$) follows directly from
  by Theorem~\ref{thm:one_step}.

  Assume the induction hypothesis ~\eqref{eqn:indhyp} is true
  for some $i \in [n-1]$.
  We will prove that there exists an $l \in [n] \setminus \{\pi(j) : j
  \in [i]\}$ that satisfies
  \begin{flalign}\label{eqn:target}
    |\lambda_l -\hat \lambda_{i+1}| \le 2 \eps, \quad \norm{\bm v_l -
    \hat{\bm v}_{i+1}}{}\le 20 \eps/ \lambda_l.
  \end{flalign}
  To simplify notation, we assume without loss of generality
  (by renumbering indices) that $\pi(j) = j$ for each $j \in [i]$.
  Let $\hat{\bm x} = \sum_{i\in[n]} x_i \bm v_i := \hat{\bm
  v}_{i+1}$ and $\hat \lambda := \hat \lambda_{i+1}$, and further
  assume without loss of generality (again by renumbering indices) that
  \[
    \lambda_{i+1} |x_{i+1}|^{p-2} \ge \lambda_{i+2} |x_{i+2}|^{p-2}
    \ge \dotsb \ge \lambda_n |x_n|^{p-2}.
  \]
  In the following, we will show that $l=i+1$ is an index satisfying
  \eqref{eqn:target}.
  We use the assumption that
  \begin{equation}
    \eps < \min\Braces{
      \frac18 ,\,
      \frac1{2.5+10c_1} ,\,
      \frac1{10(40c_2n)^{1/(p-1)}}
    } \cdot \lambda_{\min}
    \label{eqn:epscond}
  \end{equation}
  (which holds with a suitable choice of $c_0$ in the theorem
  statement).
  Here, $c_1$ and $c_2$ are the constants from Lemma~\ref{lem:main}
  when $\hat\eps = 10\eps$.
  It can be verified that \eqref{eqn:indhyp} implies that the
  conditions for Lemma~\ref{lem:main} are satisfied with this value of
  $\hat\eps$.

  Recall that $\hat\lambda = \wh{\mcb T}_i \hat{\mb x}^{\tensor p}$,
  where
  \[
    \wh{\mcb T}_i
    = \wh{\mcb T} -\sum_{j=1}^i \hat{\lambda}_j \hat{\bm v}_j^{\tensor p}
    = \sum_{j=i+1}^{n} \lambda_j \bm v_j^{\tensor p} + \mcb E +
    \sum_{j=1}^i \mb \Delta_j .
  \]
  We now bound $\hat\lambda$ from above and below.
  For the lower bound, we use \eqref{eqn:lem_S2} from Lemma
  \ref{lem:main} to obtain
  \begin{flalign}
    \hat \lambda
    & = \wh{\mcb T}_i \hat{\bm x}^{\tensor p}
    \ge \max_{j \in [n] \setminus [i]} \wh{\mcb T}_i {\bm v_j}^{\tensor p}
    \ge \lambda_{\max,i} - \eps - c_2n
    \Parens{
      \frac{10\eps}{\lambda_{\min}}
    }^{p-1} \eps
    \ge \lambda_{\max,i} - 1.25\eps
    \label{eqn:thm_low}
  \end{flalign}
  where $\lambda_{\max,i} := \max_{j \in [n] \setminus [i]} \lambda_j$
  and $\lambda_{\min} := \min_{j \in [n]} \lambda_j$; the final
  inequality uses the conditions on $\eps$ in~\eqref{eqn:epscond}.
  For the upper bound, we have
  \begin{flalign}
    \hat \lambda = \wh{\mcb T}_i \hat{\bm x}^{\tensor p}
    & = \sum_{j=i+1}^{n} \lambda_i x_i^p + \mcb E\hat{\bm x}^{\tensor p}
    + \sum_{j=1}^i \mb \Delta_j\hat{\bm x}^{\tensor p}
    \nonumber \\
    & \le \sum_{j=i+1}^n \lambda_j x_j^p
    + \eps
    + 10c_1 \sum_{j=1}^i |x_j|^{p-1} \eps
    + 10c_2 n \Parens{ \frac{10\eps}{\lambda_{\min}} }^{p-1}\eps
    \nonumber \\
    & \le \lambda_{i+1} |x_{i+1}|^{p-2} \sum_{j=i+1}^n x_j^2
    + \eps
    + 10c_1 \eps \sum_{j=1}^i x_j^2
    + 10c_2 n \Parens{ \frac{10\eps}{\lambda_{\min}} }^{p-1}\eps
    \nonumber \\
    & \le \max\Braces{ \lambda_{i+1} |x_{i+1}|^{p-2} ,\, 10c_1 \eps }
    + 1.25\eps
    .
    \label{eqn:thm_up}
  \end{flalign}
  The first inequality above
 follows from~\eqref{eqn:lem_S_all} in
  Lemma~\ref{lem:main}; the third inequality uses the fact that
  $\sum_{j=1}^n x_j^2 = 1$ as well as the conditions on $\eps$
  in~\eqref{eqn:epscond}.
  If the $\max$ is achieved by the second argument $10c_1\eps$, then
  combining~\eqref{eqn:thm_low} and~\eqref{eqn:thm_up} gives
  \[
    (2.5+10c_1)\eps \geq \lambda_{\max,i} \geq \lambda_{\min} ,
  \]
  a contradiction of~\eqref{eqn:epscond}.
  Therefore the $\max$ in~\eqref{eqn:thm_up} must be achieved by
  $\lambda_{i+1} |x_{i+1}|^{p-2}$, and hence
  combining~\eqref{eqn:thm_low} and~\eqref{eqn:thm_up} gives
  \begin{equation*}
    \lambda_{i+1}|x_{i+1}|^{p-2}
    \geq \lambda_{\max,i} - 2.5\eps
    \quad\text{and}\quad
    |\hat\lambda - \lambda_{i+1}| \leq 1.25\eps
    .
  \end{equation*}
  This in turn implies that
  \begin{equation}
    |x_{i+1}| \geq |x_{i+1}|^{p-2}
    \geq 1 - \frac{2.5\eps}{\lambda_{i+1}} , \quad
    \lambda_{i+1}
    \geq \lambda_{\max,i} - 2.5\eps , \quad \text{and} \quad
    x_{i+1}^2
    \geq x_{i+1}^{p-1}
    \geq 1 - \frac{5\eps}{\lambda_{i+1}} .
    \label{eqn:weakbound}
  \end{equation}
  Thus, we have shown that $\hat{\bm x}$ is indeed coherent with $\hat{\bm v}_{i+1}$. Next, we will sharpen the bound for $\norm{\hat{\bm x} - \hat{\bm v}_{i+1}}{}$ by considering the first order optimality condition.

  Since $\hat{\bm x} \in \arg \min_{\norm{\bm x}{2} = 1} \wh{\mcb T}_i
  \bm x^{\tensor p}$, a first-order optimality condition similar to
  \eqref{eqn:1st_order_opt} implies $\hat \lambda = \wh{\mcb T}_i
  \hat{\bm x}^{\tensor p}$.
  Thus
  \begin{flalign*}
    \hat \lambda \hat{\bm x} = \wh{\mcb T}_i \hat{\bm x}^{\tensor p-1} &= \left(\sum_{j=i+1}^n \lambda_j \bm v_j^{\tensor p} + \mcb E + \sum_{j=1}^i \mb \Delta_j\right) \hat{\bm x}^{\tensor p-1} \\
    &= \lambda_{i+1} x_{i+1}^{p-1}\bm v_{i+1}+ \sum_{j=i+2}^n \lambda_j x_j^{p-1} \bm v_j + \mcb E \hat{\bm x}^{\tensor p-1} + \sum_{j=1}^i \mb \Delta_j \hat{\bm x}^{\tensor p-1}.
  \end{flalign*}
  Therefore
  \begin{align}
    \lefteqn{
      \norm{\lambda_{i+1}(\hat{\bm x} - \bm v_{i+1})}{2}
    } \nonumber \\
    & = \norm{(\lambda_{i+1} - \hat \lambda) \hat{\bm x} + ( \hat \lambda \hat{\bm x} - \lambda_{i+1}\bm v_{i+1})}{2} \nonumber \\
    & = \norm{(\lambda_{i+1} - \hat \lambda) \hat{\bm x} + \lambda_{i+1}(x_{i+1}^{p-1}-1)\bm v_{i+1} + \sum_{j=i+2}^n \lambda_j x_j^{p-1} \bm v_j + \mcb E \hat{\bm x}^{\tensor p-1} + \sum_{j=1}^i \mb \Delta_j \hat{\bm x}^{\tensor p-1}}{2} \nonumber \\
    & \le |\lambda_{i+1} - \hat \lambda| + \lambda_{i+1} |x_{i+1}^{p-1}-1| + \norm{\sum_{j=i+2}^n \lambda_j x_j^{p-1} \bm v_j}{2} + \norm{ \mcb E \hat{\bm x}^{\tensor p-1}}{2} + \norm{\sum_{j=1}^i \mb \Delta_j \hat{\bm x}^{\tensor p-1}}{2}. \label{eqn:thm_sum}
  \end{align}
  For the third term in \eqref{eqn:thm_sum},
  we use the fact that $|x_{i+2}| \leq \sqrt{1-x_{i+1}^2}$, the bounds
  from~\eqref{eqn:weakbound} and the conditions on $\eps$
  in~\eqref{eqn:epscond} to obtain
  \begin{align}
     \norm{\sum_{j=i+2}^n \lambda_j x_j^{p-1} \bm v_j}{2}
     & = \Parens{
      \sum_{j=i+2}^n \lambda_j^2 x_j^{2p-2}
     }^{1/2}
     \nonumber \\
     & \le \lambda_{i+2}|x_{i+2}|^{p-2}\sqrt{1-x_{i+1}^2}
     \nonumber \\
     & \le \lambda_{i+2}(1-x_{i+1}^2)
     \nonumber \\
     & \le \lambda_{\max,i} \frac{5\eps}{\lambda_{i+1}}
     \nonumber \\
     & \le \frac{5\eps}{1 - 2.5\eps/\lambda_{\max,i}}
     \nonumber \\
     & \le 7.5\eps
     .
     \label{eqn:third_term}
  \end{align}
  For the last term in \eqref{eqn:thm_sum}, we
  use~\eqref{eqn:lem_S_all} from Lemma~\ref{lem:main} and the
  conditions on $\eps$ in~\eqref{eqn:epscond} to get
  \begin{flalign}
    \norm{\sum_{j=1}^i \mb \Delta_j \hat{\bm x}^{\tensor p-1}}{2}
    & \le 10c_1 \sum_{j=1}^n |x_j|^{p-1} \eps + 10c_2 n
    \Parens{\frac{10\eps}{\lambda_{\min}}}^{p-1} \eps  \nonumber \\
    & \le 10c_1 (1 - x_{i+1}^2)\eps
    + 10c_2 n \Parens{\frac{10\eps}{\lambda_{\min}}}^{p-1} \eps
    \nonumber \\
    & \le \frac{50c_1}{\lambda_{i+1}}\eps^2
    + 0.25\eps
    \nonumber \\
    & \le 5.25\eps
    .
    \label{eqn:last_term}
  \end{flalign}
  Therefore, substituting \eqref{eqn:weakbound}, \eqref{eqn:last_term}
  and $\norm{\mcb E}{} \le \eps$ into \eqref{eqn:thm_sum} gives
  \[
    \norm{\lambda_{i+1}(\hat{\bm x} - \bm v_{i+1})}{2} \le 20 \eps .
%    \qedhere
  \]
\end{proof}
\subsection{Stability of Full Decomposition}
Theorem \ref{thm:main} states a (perhaps unexpected) phenomenon that the approximation errors do not accumulate with iteration number, whenever the perturbation error is small enough. In this subsection, we numerically corroborate this fact.

We generate nearly symmetric orthogonally decomposable tensors
$\wh{\mcb T} = \mcb T + \mcb E \in \reals^{10\times 10 \times 10}$ as follows.
We construct the underlying symmetric orthogonally decomposable tensor $\mcb
T$ as the diagonal tensor with all diagonal entries equal to one,
i.e., $\mcb T = \sum_{i=1}^{10} \bm e_i^{\tensor 3}$ (where $\bm e_i$
is the $i$-th coordinate basis vector).
The perturbation tensor $\mcb E$ is generated under three random
models with $\sigma = 0.01$:
\begin{description}
  \item[Binary:]
    independent entries $\mc E_{i,j,k} \in \{\pm\sigma\}$ uniformly at
    random;

  \item[Uniform:]
    indepedent entries $\mc E_{i,j,k} \in [-2\sigma,2\sigma]$
    uniformly at random;

  \item[Gaussian:]
    independent entries $\mc E_{i,j,k} \sim \mc N(0,\sigma^2)$.

\end{description}

For each random model, we generate $500$ random instances, and apply
Algorithm~\ref{alg:rank-1} to each $\wh{\mcb T}$ to obtain approximate
pairs $\{(\hat \lambda_i, \hat{\bm v}_i)\}_{i\in[10]}$.
Again, we use \textsf{GloptiPoly~3} to solve the polynomial
optimization problem in Algorithm~\ref{alg:rank-1}.

In Figure~\ref{fig:full_deflation}, we plot the mean and the standard
deviation of the approximation errors for $\hat\lambda_i$ and $\hat{\bm
v}_i$ from the $500$ random instances (for each $i\in [10]$).
These indeed do not appear to grow or accumulate as the
iteration number increases.
This is consistent with our results in Theorem~\ref{thm:main}.

\begin{figure*}
  \begin{center}
    \setlength{\tabcolsep}{0in}
    \begin{tabular}{cc}
      \includegraphics[height=1.5in]{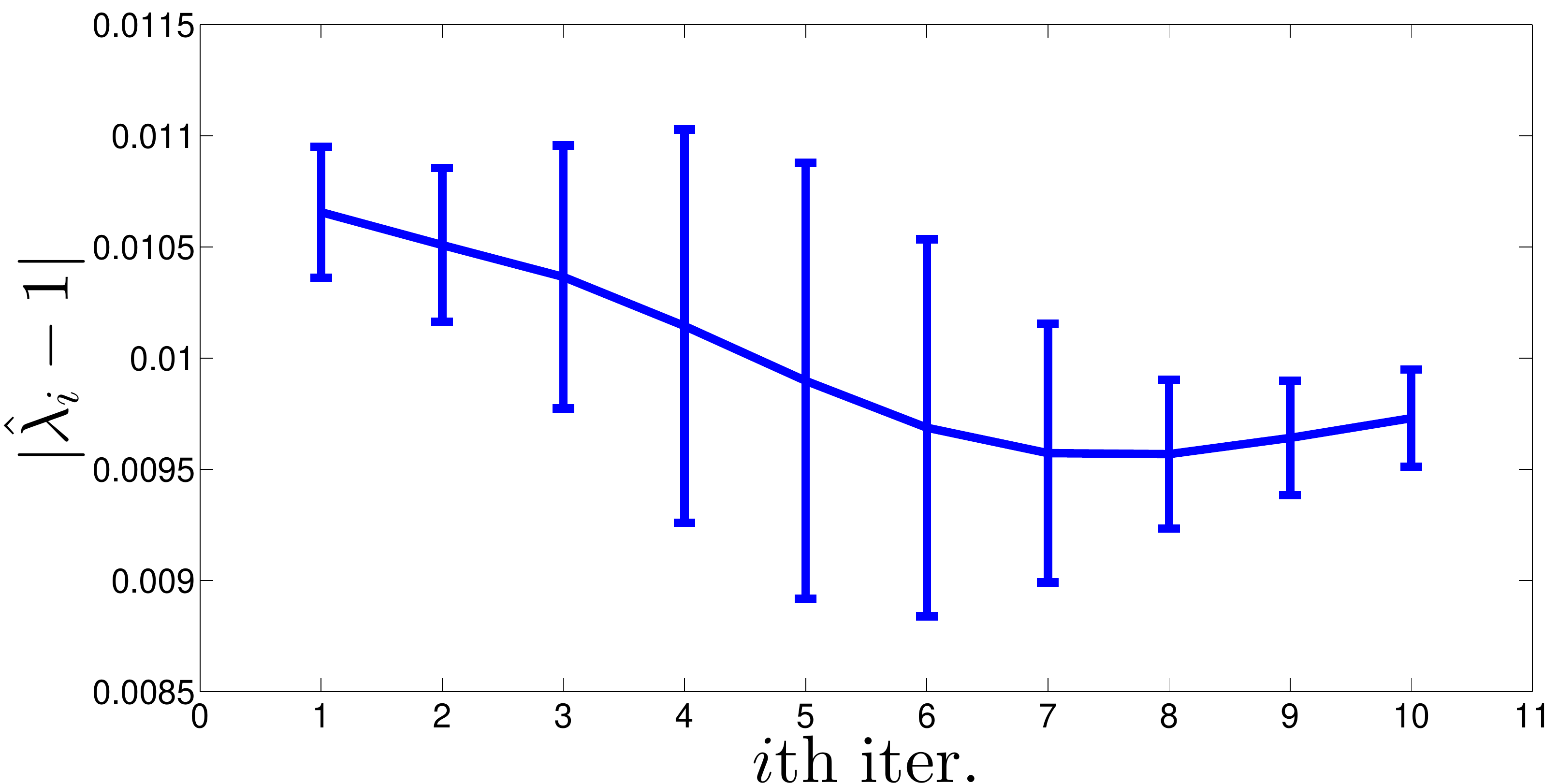} \hspace{5mm}&\hspace{5mm}
      \includegraphics[height=1.5in]{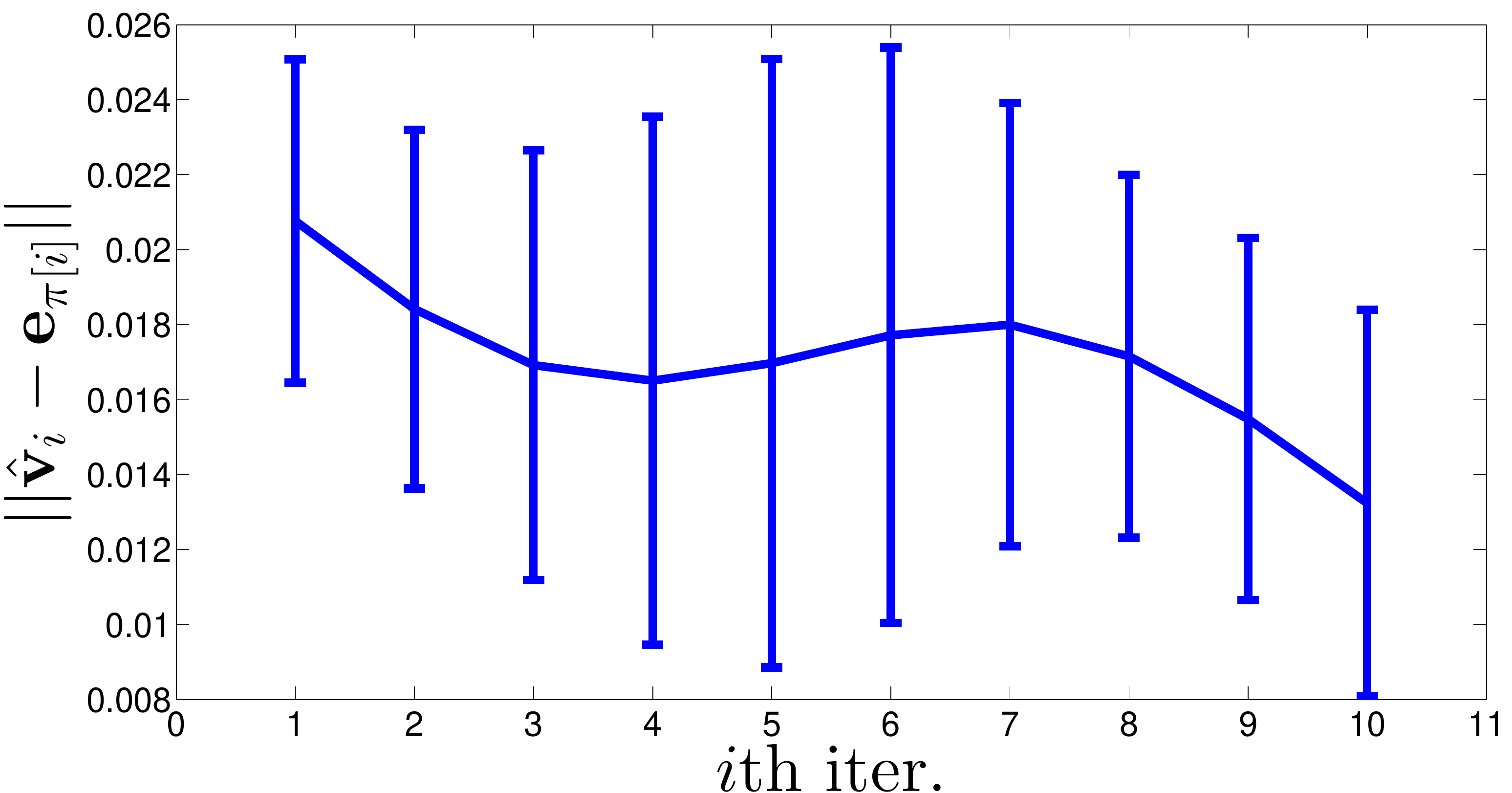}
      \\
      \multicolumn{2}{c}{\emph{Perturbation: binary}}
      \\
      \\
      \\
      \includegraphics[height=1.5in]{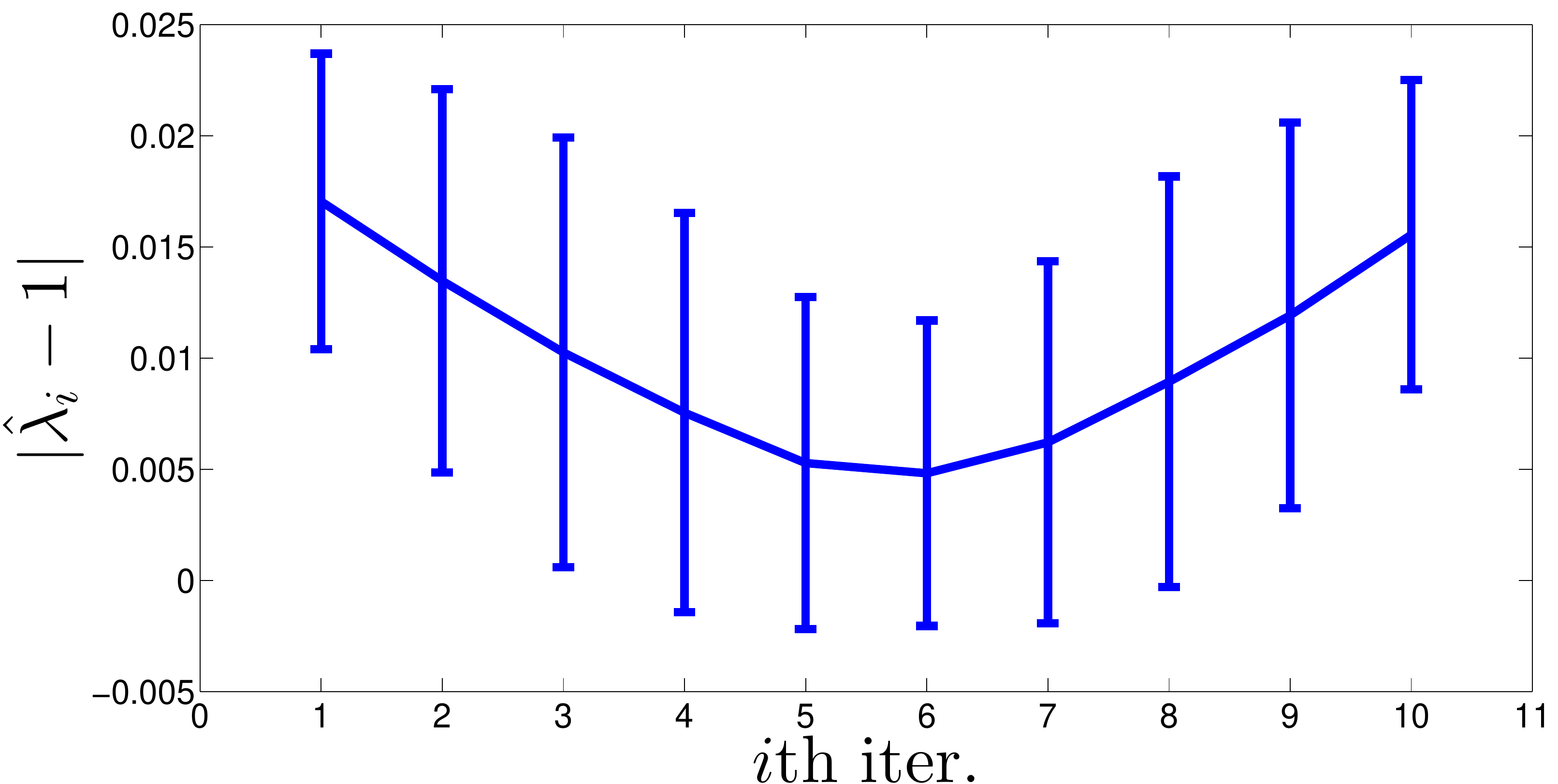} \hspace{5mm}&\hspace{5mm}
      \includegraphics[height=1.5in]{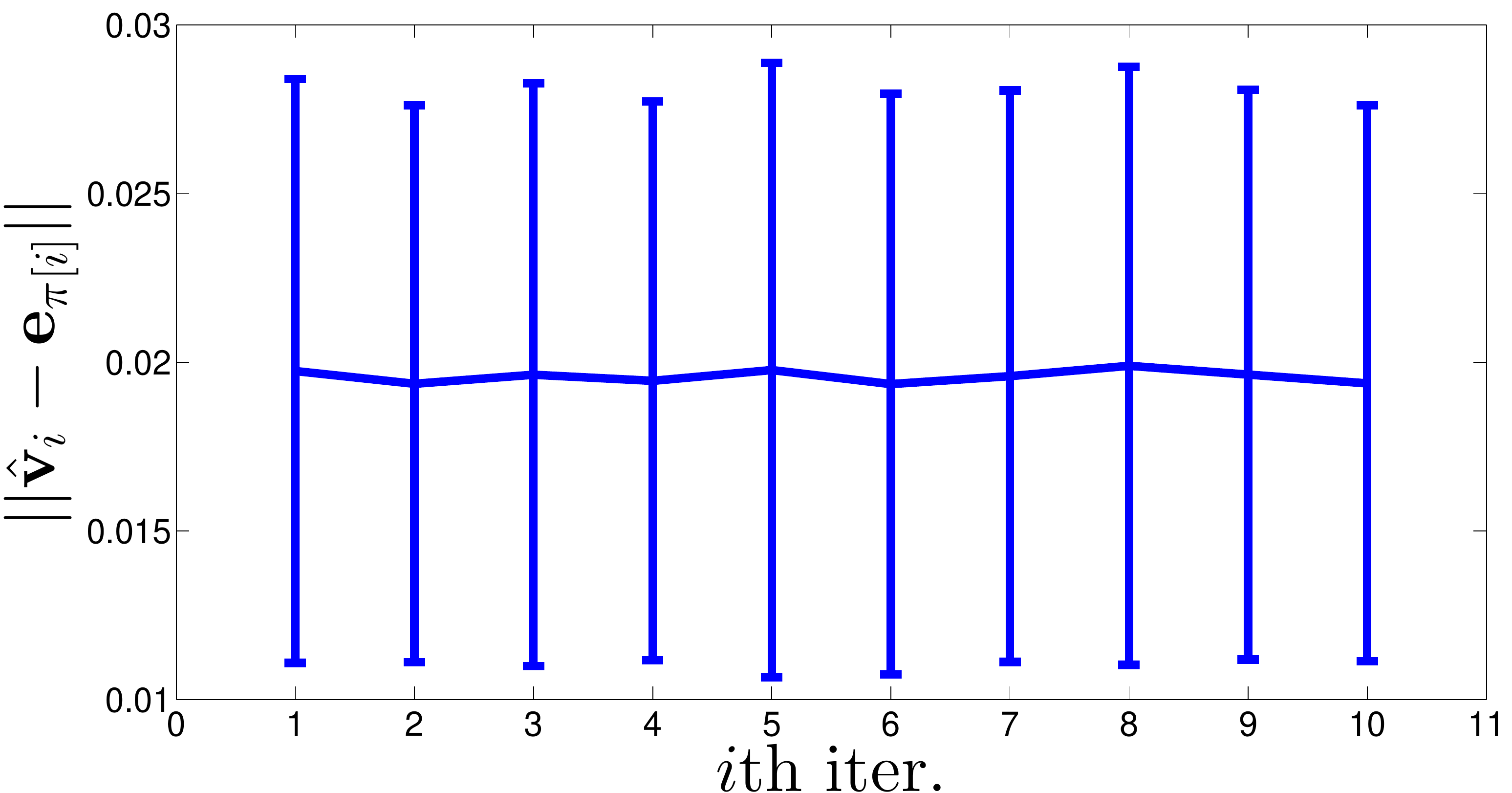}
      \\
      \multicolumn{2}{c}{\emph{Perturbation: uniform}}
      \\
      \\
      \\
      \includegraphics[height=1.5in]{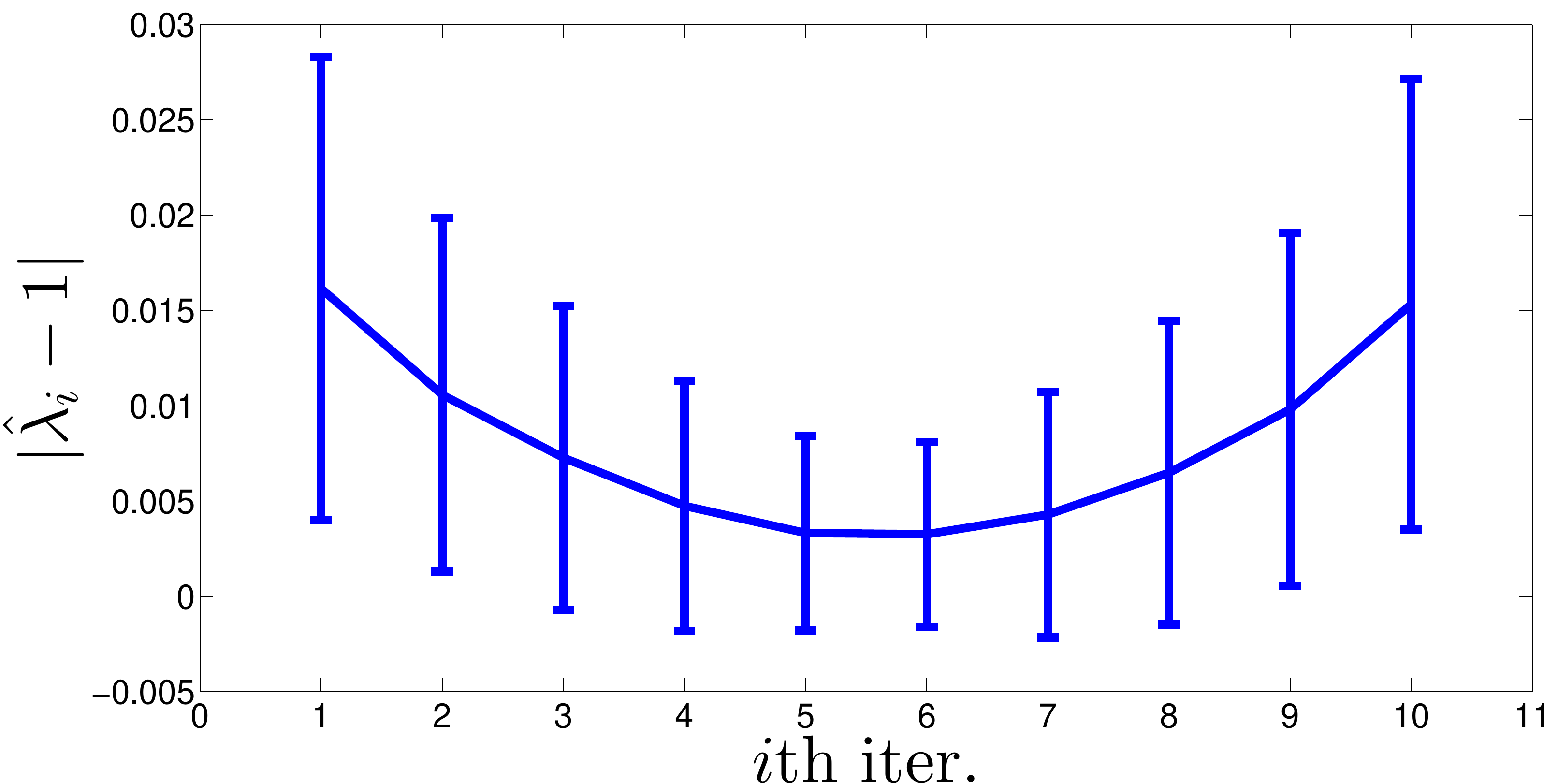}
      \hspace{5mm}&\hspace{5mm}
      \includegraphics[height=1.5in]{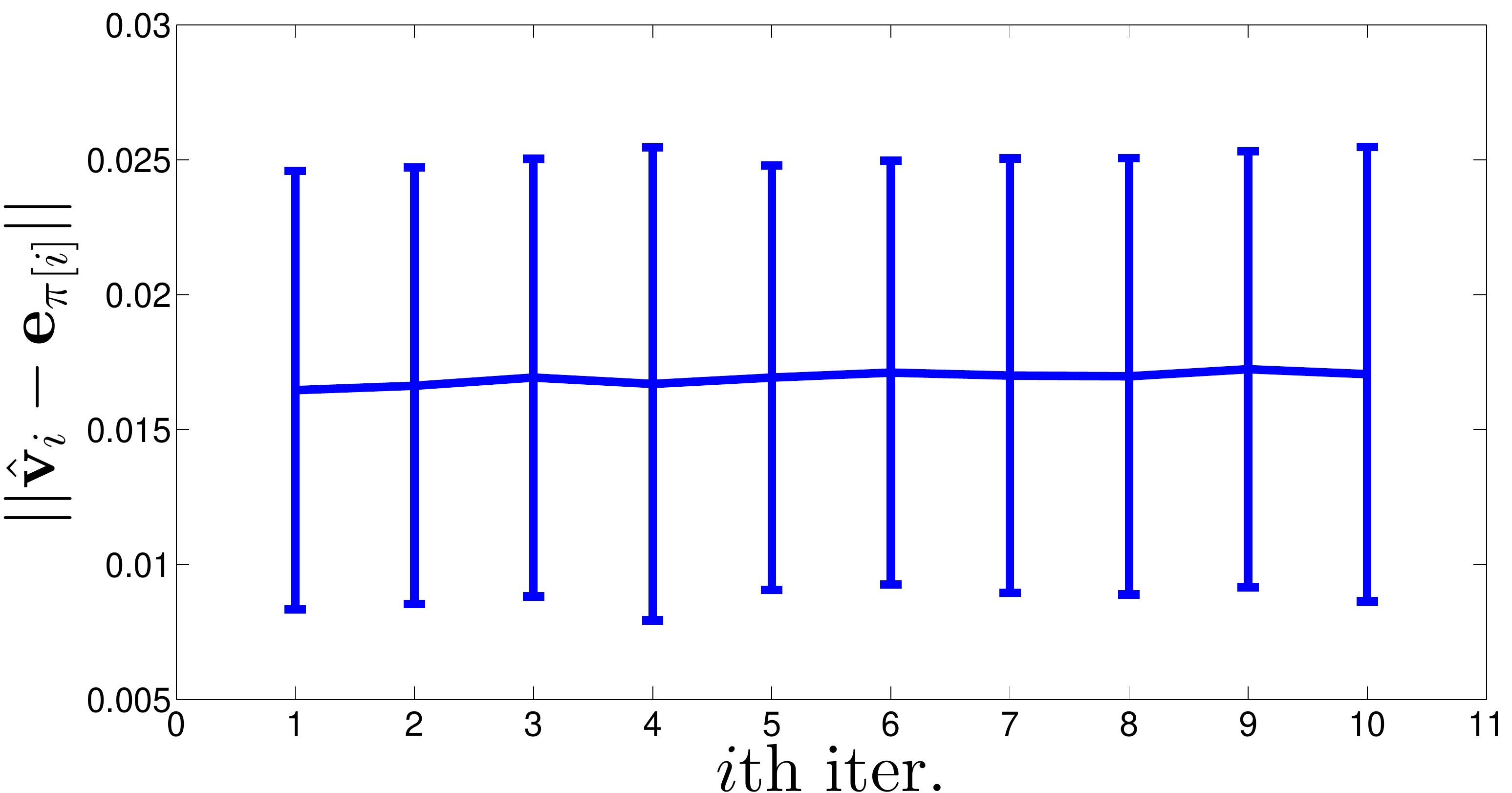}
      \\
      \multicolumn{2}{c}{\emph{Perturbation: Gaussian}}
      \\
      \\
      \\
    \end{tabular}
\caption{%
  \textbf{Approximation Errors of Algorithm~\ref{alg:rank-1}.}
  For each vertical bar over the iteration index $i\in [10]$, the midpoint is
  the mean of the approximation errors of $\hat \lambda_i$ (left) and
  $\hat{\bm v}_i$ (right), computed over 500 randomly generated instances.
  The error bars extend to two standard deviations above and below the
  mean.%
}
\label{fig:full_deflation}
  \end{center}
\end{figure*}

\subsection{When $p$ is Even}
\label{sec:even}

We now briefly discuss the case where the order of the tensor is even,
i.e., $p\ge 4$ is an even integer.

Let $\wh{\mcb T} := \mcb T + \mcb E \in \Tensor^p \reals^n$, where $\mcb T$ is a symmetric tensor with orthogonal decomposition $\mcb T = \sum_{i=1}^n \lambda_i \bm v_i^{\tensor p}$, where $\set{\bm v_1, \bm
v_2, \dotsc, \bm v_n}$ is an orthonormal basis of $\reals^n$, $\lambda_i \neq 0$ for all $i\in [n]$, and $\mcb E$ is a symmetric tensor with operator norm $\eps:=\norm{\mcb E}{}$. Note that unlike the case when $p$ is odd, we cannot assume $\lambda_i > 0$ for all $i \in [n]$, and correspondingly, line \ref{line:rank_one_approx} in Algorithm \ref{alg:rank-1} now becomes
\begin{flalign*}
  \hat{\bm v}_i \in \argmax_{\norm{\bm v}{}=1} \left|\wh{\mcb T}_{i-1} \bm
  v^{\tensor p}\right| =  \argmax_{\norm{\bm v}{}=1} \; \max\left\{\wh{\mcb T}_{i-1} \bm
  v^{\tensor p}, -\wh{\mcb T}_{i-1} \bm
  v^{\tensor p}\right\}, \quad
  \hat \lambda_i = \wh{\mcb T}_{i-1} \hat{\bm v}_i^{\tensor p} .
\end{flalign*}
Nevertheless, the pair $(\hat \lambda_i, \hat{\bm v}_i)$ still satisfies the first-order optimality condition $\hat \lambda_i \hat{\bm v}_i = \wh{\mcb T} \hat{\bm v}_i^{\tensor p-1}$.

Our proof for Theorem \ref{thm:main} can be easily modified and leads to the following result: there exists a positive constant $\hat c_0 = \hat c_0(p) >0$ such that whenever $\eps \le \hat c_0 \left(\min_{i\in [n]} |\lambda_i|\right)/n^{1/(p-1)}$,   there exists a permutation $\pi$ on $[n]$ such that
  \[
    |\lambda_{\pi(j)} - \hat{\lambda}_j| \le 2\eps, \qquad  \min\left\{\norm{\bm
    v_{\pi(j)}-\hat{\bm v}_j}{}, \norm{\bm
    v_{\pi(j)}+\hat{\bm v}_j}{} \right\}\le 20\eps/|\lambda_{\pi(j)}|  , \quad
    \forall j \in [n].
  \]

\section{Conclusion}
This paper sheds light on a problem at the intersection of numerical
linear algebra and statistical estimation, and our results draw upon
and enrich the literature in both areas.

From the perspective of numerical linear algebra, SROA was previously
only known to exactly recover the symmetric canonical decomposition of
an orthogonal decomposable tensor.
Our results show that it can robustly recover (approximate) orthogonal
decompositions even when applied to nearly SOD tensors; this
substantially enlarges the applicability of SROA.

Previous work on statistical estimation via orthogonal tensor
decompositions considered the specific randomized power iteration
algorithm of~\citet{JMLR:v15:anandkumar14b}, which has been
successfully applied in a number of
contexts~\cite{chaganty2013linear,ZHPA13-contrast,azar2013bandit,huang2013hmm,AGHK14-community,doshi2014graph}.
Our results provide formal justification for using other rank-one
approximation methods in these contexts, and it seems to be quite beneficial, in terms of sample complexity and statistical efficiency, to use more sophisticated methods. Specifically, the perturbation error $\norm{\mcb E}{}$ that can be tolerated is relaxed from power iteration's $O(1 / n)$ to $O(1/\sqrt[p-1]{n})$.
In future work, we plan to empirically investigate these potential
benefits in a number of applications.

We also note that solvers for rank-one tensor approximation often lack rigorous runtime
or error analyses, which is not surprising given the computational
difficulty of the problem for general tensors~\citep{hillar2009most}.
However, tensors that arise in applications are often more
structured, such as being nearly SOD.
Thus, another promising future research direction is to sidestep computational
hardness barriers by developing and analyzing methods for such specially structured
tensors (see also \cite{JMLR:v15:anandkumar14b, barak2014dictionary} for ideas along this line).

\section*{Acknowledgements}
Daniel Hsu acknowledges support from a Yahoo ACE award.
Donald Goldfarb acknowledges support from NSF Grants DMS-1016571 and CCF-1527809.

\appendix
\section{Proof of Theorem~\ref{thm:matrix_pert}}
\label{sec:matrix_pert}

Since $\wh{\mb M}$ is symmetric, it has an eigenvalue decomposition
$\sum_{i=1}^n \hat\lambda_i \hat{\mb v}_i \hat{\mb v}_i^\t$, where
$|\hat\lambda_1| \ge |\hat\lambda_2| \ge \dotsb \ge |\hat\lambda_n|$
and $\{ \hat{\bm v}_1, \hat{\bm v}_2, \dotsc, \hat{\bm v}_n \}$ are
orthonormal.
It is straightforward to obtain:
\[
  \hat\lambda = \hat\lambda_1
  \quad\text{and}\quad
  \wh{\mb M} \hat{\mb x} = \hat\lambda \hat{\mb x} .
\]
By Weyl's inequality~\cite{weyl1912asymptotische},
\[
  |\hat\lambda - \lambda_1| \leq \norm{\mb E}{} = \eps .
\]
To bound $\innerprod{\hat{\bm x}}{\bm v_1}^2$, we employ an
argument very similar to one from~\cite{davis1970rotation}.
Observe that
\[
  \norm{\mb M \hat{\mb x} - \lambda_1 \hat{\mb x}}{}^2
  = \norm{(\hat\lambda - \lambda_1) \hat{\mb x} - \mb E\hat{\mb
  x}}{}^2 \leq (|\hat\lambda - \lambda_1| \norm{\hat{\mb x}}{}
  + \norm{\mb E\hat{\mb x}}{})^2
  \leq 4\eps^2
  .
\]
Moreover,
\[
  \mb M \hat{\mb x} - \lambda_1 \hat{\mb x}
  = \sum_{i=1}^n (\lambda_i - \lambda_1)
  \innerprod{\mb v_i}{\hat{\mb x}} \mb v_i
  = \sum_{i=2}^n (\lambda_i - \lambda_1)
  \innerprod{\mb v_i}{\hat{\mb x}} \mb v_i
  ,
\]
and therefore
\[
  \norm{\mb M \hat{\mb x} - \lambda_1 \hat{\mb x}}{}^2
  =
  \sum_{i=2}^n (\lambda_i - \lambda_1)^2
  \innerprod{\mb v_i}{\hat{\mb x}}^2
  \geq
  \gamma^2
  \sum_{i=2}^n
  \innerprod{\mb v_i}{\hat{\mb x}}^2
  =
  \gamma^2
  (1 - \innerprod{\mb v_i}{\hat{\mb x}}^2)
  .
\]
Combining the upper and lower bounds on $\norm{\mb M \hat{\mb x} -
\lambda_1 \hat{\mb x}}{}^2$ gives $\innerprod{\hat{\bm x}}{\bm v_1}^2
\geq 1 - (2\eps/\gamma)^2$ as claimed. \qed

\section{Proof of Lemma~\ref{lem:main}} \begin{proof}
  The lemma holds trivially if $\hat\eps = 0$.
  So we may assume $\hat\eps>0$.
  Therefore, for every $i \in S_1$, we have $|x_i| \ge 4\hat \eps/\lambda_i>0$.
  Let $c_i := \innerprod{\bm v_i}{\hat{\bm v}_i}$,
  $\bm w_i := (\hat{\bm v}_i - c_i \bm v_i)/\norm{\hat{\bm v}_i - c_i
  \bm v_i}{2}$, and
  $y_i := \innerprod{\bm w_i}{\bm x}$, so
  \[
    \hat{\bm v}_i = c_i \bm v_i + \sqrt{1-c_i^2} \bm w_i
    \quad \text{and} \quad
    \innerprod{\hat{\bm v}_i}{\bm x} = c_i x_i + \sqrt{1-c_i^2}y_i .
  \]

  We first establish a few inequalities that will be frequently used
  later.
  Since $|\lambda_i - \hat \lambda_i| \le \hat \eps \le \lambda_i/2$,
  one has $\hat \eps/ \lambda_i \le 1/2$, and $1/2 \le \hat \lambda_i
  / \lambda_i \le 3/2$.
  Also, since $c_i \ge 1 - 2(\hat \eps/ \lambda_i)^2\ge 1/2$,
  \begin{flalign*}
    \sqrt{1-c_i^2}
    & \le \sqrt{1- \left(1 - 2\left(\hat \eps/ \lambda_i\right)^2 \right)^2}
    = \sqrt{4 \left( \hat \eps/\lambda_i \right)^2\left(1 - \left(
    \hat \eps/\lambda_i \right)^2 \right)}
    \le 2\hat \eps/\lambda_i .
  \end{flalign*}

  For each $i\in S$,
  \begin{align*}
    &\mb{\Delta}_i \bm x^{\otimes p-1} \\
    & =
    \Parens{
      \lambda_i \bm v_i^{\otimes p}
      - \hat{\lambda}_i \hat{\bm v}_i^{\otimes p}
    }
    \bm x^{\otimes p-1} \nonumber \\
    & = \lambda_i x_i^{p-1} \bm v_i - \hat \lambda_i
    \innerprod{\hat{\bm v}_i}{\bm x}^{p-1} \hat{\bm v}_i
    \\
    & = \lambda_i x_i^{p-1} \bm v_i - \hat \lambda_i
    \Parens{
      c_i x_i + \sqrt{1-c_i^2}y_i
    }^{p-1}
    \Parens{
      c_i \bm v_i + \sqrt{1-c_i^2} \bm w_i
    }
    \\
    & =
    \Parens{
      \lambda_i x_i^{p-1} - \hat \lambda_i c_i
      \Parens{
        c_i x_i + \sqrt{1-c_i^2} y_i
      }^{p-1}
    }
    \bm v_i -
    \Parens{
      \hat \lambda_i \sqrt{1-c_i^2}
      \Parens{
        c_i x_i + \sqrt{1-c_i^2} y_i
      }^{p-1}
    }
    \bm w_i.
  \end{align*}
  Therefore, due to the orthonormality of $\{\bm v_i\}_{i\in[n]}$ and
  the triangle inequality, for each $j \in \{1,2\}$,
  \begin{equation}
    \begin{aligned}
      \norm{\sum_{i \in S_j} \mb{\Delta}_i \bm x^{\otimes p-1}}{2}
      & \le
      \Parens{
        \sum_{i \in S_j}
        \Parens{
          \lambda_i x_i^{p-1} - \hat \lambda_i c_i
          \Parens{
            c_i x_i + \sqrt{1-c_i^2} y_i
          }^{p-1}
        }^2
      }^{1/2}
      \\
      & \quad
      + \sum_{i \in S_j}
      \Abs{
        \hat \lambda_i \sqrt{1-c_i^2}
        \Parens{
          c_i x_i + \sqrt{1-c_i^2} y_i
        }^{p-1}
      }
      .
    \end{aligned}
    \label{eqn:S_j}
  \end{equation}

  We now prove \eqref{eqn:deflation_S1}.
  For any $i \in S_1$, since $x_i \neq 0$, we may
  write~\eqref{eqn:S_j} as
  \begin{flalign}
    \begin{aligned} \label{eqn:bound_S1}
      \norm{\sum_{i \in S_1} \mb{\Delta}_i \bm x^{\otimes p-1}}{2}
      & \le
      \Parens{
        \sum_{i \in S_1} x_i^{2p-4}
        \Parens{
          \lambda_i x_i - \hat\lambda_i x_i c_i^p
          \Parens{
            1+\sqrt{\frac{1-c_i^2}{c_i^2}}\frac{y_i}{x_i}
          }^{p-1}
        }^2
      }^{1/2}  \\
      & \quad
      + \sum_{i \in S_1}
      \Abs{
        \hat \lambda_i x_i^{p-1} c_i^{p-1}
        \sqrt{1-c_i^2}
        \Parens{
          1+\sqrt{\frac{1-c_i^2}{c_i^2}}\frac{y_i}{x_i}
        }^{p-1}
      }
      .
    \end{aligned}
  \end{flalign}
  Observe that
  \[
    \Abs{
      \sqrt{\frac{1-c_i^2}{c_i^2}} \frac{y_i}{x_i}
    }
    \le \frac{\sqrt{1-c_i^2}}{|c_i|}\frac{1}{|x_i|} \le \frac{4 \hat
    \eps}{\lambda_i |x_i|}
    \le 1
  \]
  because $|c_i| \ge 1/2$ and $\sqrt{1 - c_i^2} \ge 2 \hat \eps/\lambda_i$.
  Moreover, since $1+ (p-1)z \le (1+z)^{p-1} \le 1 + (2^{p-1}-1)z$ for
  any $z \in [0,1]$,
  \begin{flalign}\label{eqn:S1_first_p-1}
    \Abs{
      \Parens{
        1 + \sqrt{\frac{1-c_i^2}{c_i^2}} \frac{y_i}{x_i}
      }^{p-1} - 1
    }
    \le (2^{p-1}-1) \frac{4 \hat \eps}{\lambda_i |x_i|} = (2^{p+1} -4)
    \frac{\hat \eps}{\lambda_i |x_i|}
    .
  \end{flalign}
  Therefore,
  \begin{flalign}
    \Abs{
      \lambda_i x_i - \hat{\lambda}_ix_i c_i^p
      \Parens{
        1+ \sqrt{\frac{1-c_i^2}{c_i^2}} \frac{y_i}{x_i}
      }^{p-1}
    }
    & \le
    |\lambda_i x_i - \hat{\lambda}_i x_i|
    + \hat \lambda_i |x_i|
    \Abs{
      1 -  c_i^p
      \Parens{
        1+ \sqrt{\frac{1-c_i^2}{c_i^2}} \frac{y_i}{x_i}
      }^{p-1}
    }
    \nonumber \\
    & \le \hat \eps + \frac{\hat \lambda_i |x_i|}{\lambda_i}
    \Parens{
      (2^{p+1}-4)\frac{ \hat \eps}{|x_i|} + p \hat \eps + (2^{p+1} -4)
      \frac{ \hat \eps}{|x_i|} \frac{p \hat \eps}{\lambda_i}
    }
    \nonumber \\
    & \le \hat \eps + \frac{3}{2}
    \Parens{
      (2^{p+1}-4) + p + (2^{p-1}-1) p
    }
    \hat \eps
    \nonumber \\
    & \le 2^{p+1}p \hat \eps.
    \label{eqn:S1_first_dev}
  \end{flalign}
  The second inequality above is obtained using the inequality
  $|(1+a)(1+b)-1| \le |a|+|b|+|ab|$ for any $a,b \in\reals$, together
  with the inequality from~\eqref{eqn:S1_first_p-1} and the fact
  $|1-c_i^p| \le 2p(\hat \eps/\lambda_i)^2 \le p(\hat
  \eps/\lambda_i)$.
  Using the resulting inequality in~\eqref{eqn:S1_first_dev},
  the first summand in \eqref{eqn:bound_S1} can be bounded as
  \begin{equation}
    \Parens{
      \sum_{i \in S_1} x_i^{2p-2}
      \Parens{
        \lambda_i - \hat \lambda_i c_i^p
        \Parens{
          1+\sqrt{\frac{1-c_i^2}{c_i^2}}\frac{y_i}{x_i}
        }^{p-1}
      }^2
    }^{1/2}
    \le 2^{p+1}p
    \Parens{
      \sum_{i\in S_1} x_i^{2(p-2)}
    }^{1/2}
    \hat\eps
    .
    \label{eqn:S1_first}
  \end{equation}
  To bound the second summand in \eqref{eqn:bound_S1}, we have
  \begin{flalign}
    \sum_{i \in S_1}
    \Abs{
      \hat\lambda_i x_i^{p-1} c_i^{p-1} \sqrt{1-c_i^2}
      \Parens{
        1+\sqrt{\frac{1-c_i^2}{c_i^2}}\frac{y_i}{x_i}
      }^{p-1}
    }
    & \le \sum_{i \in S_1}
    \Abs{
      \hat \lambda_i x_i^{p-1} \sqrt{1-c_i^2}
      \Parens{
        1+\frac{\sqrt{1-c_i^2}}{c_i} \frac{1}{|x_i|}
      }^{p-1}
    }
    \nonumber \\
    & \le \sum_{i \in S_1}
    \Abs{
      \hat \lambda_i x_i^{p-1} \frac{2 \hat \eps}{\lambda_i}
      \Parens{
        1+ \frac{4\hat \eps}{\lambda_i |x_i|}
      }^{p-1}
    }
    \nonumber \\
    & \le 2^{p+1} \sum_{i\in S_1} |x_i|^{p-1} \hat \eps
    ,
    \label{eqn:S1_second}
    .
  \end{flalign}
  The second inequality uses the facts $c_i \ge 1/2$ and
  $\sqrt{1-c_i^2} \le 2 \hat \eps / \lambda_i$; the last inequality
  uses the facts $\hat \lambda_i/\lambda_i \le 3/2$ and $\lambda_i
  |x_i| \ge 4 \hat \eps$.
  Combining \eqref{eqn:S1_first} and \eqref{eqn:S1_second} gives the
  claimed inequality in \eqref{eqn:deflation_S1} via
  \eqref{eqn:bound_S1}.

  It remains to prove \eqref{eqn:deflation_S2}.
  For each $i \in S_2$,
  \begin{flalign*}
    \Abs{
      \lambda_i x_i^{p-1} - \hat \lambda_i c_i
      \Parens{
        c_i x_i + \sqrt{1-c_i^2} y_i
      }^{p-1}
    }
    & \le \lambda_i|x_i|^{p-1} + \hat\lambda_i
    \Parens{ |x_i|+\sqrt{1-c_i^2} }^{p-1}
    \\
    & \le \lambda_i
    \Parens{ \frac{4\hat \eps}{\lambda_i} }^{p-1}
    + \hat\lambda_i
    \Parens{
      \frac{4 \hat \eps}{\lambda_i} + \frac{2 \hat \eps}{\lambda_i}
    }^{p-1}
    \\
    & \le
    \Parens{
      4^{p-1}+\frac{3}{2} \cdot 6^{p-1}
    }
    \Parens{
      \frac{\hat \eps}{\lambda_i}
    }^{p-2}
    \hat\eps
    \\
    & \le 6^p
    \Parens{
      \frac{\hat \eps}{\lambda_i}
    }^{p-2} \hat\eps
    .
  \end{flalign*}
  The second inequality uses the facts $\sqrt{1-c_i^2} \le 2 \hat
  \eps/\lambda_i$ and $\lambda_i |x_i| < 4 \hat \eps$ for all $i\in
  S_2$; the third inequality uses the fact $\hat \lambda_i / \lambda_i
  \le 3/2$.
  Therefore
  \begin{equation}
    \Parens{
      \sum_{i \in S_2}
      \Parens{
        \lambda_i x_i^{p-1} - \hat \lambda_i c_i
        \Parens{
          c_i x_i + \sqrt{1-c_i^2} y_i
        }^{p-1}
      }^2
    }^{1/2}
    \le 6^p
    \Parens{
      \sum_{i \in S_2}
      \Parens{
        \frac{\hat \eps}{\lambda_i}
      }^{2(p-2)}
    }^{1/2}
    \hat \eps
    .
    \label{eqn:S2_first}
  \end{equation}
  Moreover,
  \begin{flalign}
    \sum_{i \in S_2}
    \Abs{
      \hat \lambda_i \sqrt{1-c_i^2}
      \Parens{
        c_i x_i + \sqrt{1-c_i^2} y_i
      }^{p-1}
    }
    & \le \sum_{i \in S_2}
    \Abs{
      \hat\lambda_i \sqrt{1-c_i^2}
      \Parens{
        |x_i| + \sqrt{1-c_i^2}
      }^{p-1}
    }
    \nonumber \\
    & \le \sum_{i \in S_2} \hat \lambda_i \frac{2 \hat \eps}{\lambda_i}
    \Parens{
      \frac{4 \hat \eps}{\lambda_i} + \frac{2 \hat \eps}{\lambda_i}
    }^{p-1}
    \nonumber \\
    & \le 3 \cdot 6^{p-1}\sum_{i \in S_2}
    \Parens{
      \frac{\hat \eps}{\lambda_i}
    }^{p-1}
    \hat\eps
    \nonumber \\
    & \le 6^p \sum_{i \in S_2}
    \Parens{
      \frac{\hat \eps}{\lambda_i}
    }^{p-1} \hat \eps
    .
    \label{eqn:S2_second}
  \end{flalign}
  Combining \eqref{eqn:S2_first} and \eqref{eqn:S2_second} establishes
  \eqref{eqn:deflation_S2} via \eqref{eqn:S_j} (with $j=2$).
\end{proof}

\bibliographystyle{ieeetr}
\bibliography{rank-1}

\end{document}